\documentclass{amsart}

\usepackage{amssymb,amsfonts}
\usepackage{amsmath}
\usepackage{mathtools}
\usepackage{hyperref}
\usepackage{xcolor}
\usepackage{tensor}
\usepackage{tikz-cd}

\usepackage{tikz}
\usetikzlibrary{shapes,snakes,arrows, chains, positioning, shapes.geometric, shapes.symbols,calc}
\usetikzlibrary{arrows.meta}

\def\centerarc[#1](#2)(#3:#4:#5)
{ \draw[#1] ($(#2)+({#5*cos(#3)},{#5*sin(#3)})$) arc (#3:#4:#5); }

\theoremstyle{plain}
\newtheorem{theorem}{Theorem}

\newtheorem{corollary}{Corollary}

\newtheorem{lemma}{Lemma}
\newtheorem{proposition}{Proposition}

\newtheorem{conjecture}{Conjecture}

\theoremstyle{definition}
\newtheorem{definition}{Definition}
\newtheorem{remark}{Remark}
\newtheorem{example}{Example}

\newtheorem{question}{Question}
\newtheorem{expectation}{Expectation}

\newcommand{\p}{\partial}

\newcommand{\C}{\mathbb{C}}
\newcommand{\R}{\mathbb{R}}
\newcommand{\CP}{\mathbb{CP}}
\newcommand{\tz}{\Tilde{z}}
\newcommand{\tw}{\Tilde{w}}
\newcommand{\tx}{\Tilde{x}}

\title{Some models for bubbling of (log) Kähler-Einstein metrics}
\author{Martin de Borbon and Cristiano Spotti}
\address{Department of Mathematical Sciences, Loughborough University, Loughborough LE11 3TU (United Kingdom), \emph{m.de-borbon@lboro.ac.uk}}
\address{Department of Mathematics, Aarhus University,  Ny Munkegade 118 8000 Aarhus  (Denmark), \emph{c.spotti@math.au.dk}}

\begin{document}

\maketitle

\begin{abstract}
    We investigate aspects of the metric bubble tree for non-collapsing degenerations of (log) Kähler-Einstein metrics in complex dimensions one and two, and further describe a conjectural higher dimensional picture. 
\end{abstract}


\section{Introduction}

In the last decade there has been an intense study of non-collapsed limits of Kähler-Einstein (KE) manifolds and their relation to algebraic geometry. For instance, results include the algebricity of non-collapsed Gromov-Hausdorff (GH) limits (\cite{donsun1, liuszek}), which led to the solution of the Yau-Tian-Donaldson conjecture in the Fano case (\cite{cds, datarszek, chensunwang, bermanboucksomjonsson}), the construction of Fano K-moduli spaces (e.g., \cite{blumxu, SSY, odaka2, LWX}), and the algebro-geometric detection of the metric tangent cone at the singularities (\cite{donsun2, chili, liwangxu}). 

When a singularity is formed, the metric tangent cone is the first approximation of the geometry of the space around such singular point. However, during the singularity formation, the metrics contract at different rate certain portions of the degenerating spaces. From a metric point of view, such phenomenon is captured by the so-called \emph{metric bubble tree}, which encodes all possible rescaled limits of a degenerating family, including the metric tangent cone at the singularity as the first non-trivial rescaled limit. In the KE case, these resulting limiting bubbles are, in general still singular, asymptotically conical Calabi-Yau (CY) spaces.

This metric picture leads us to the following natural problem: 
given a \emph{flat family of non-collapsing Kähler-Einstein spaces} and a singular point $p_0 \in \mbox{Sing}(X_0)$ 
$$\begin{tikzcd}[column sep=0pt,nodes={inner xsep=0pt,outer xsep=0pt}]
   p_0 \in X_0 & \hookrightarrow & \mathcal{X}\arrow[d]  \\
  &   & \Delta, 
\end{tikzcd}$$
how can we find all the \emph{rescaled} pointed Gromov-Hausdorff limits (i.e., \emph{bubbles}) $$\lim_{i\rightarrow \infty} (\lambda(t_i)X_{t_i}, p_{t_i})=B_{\infty}$$
at points $p_{t_i} \in X_{t_i}$ such that $p_{t_i}\rightarrow p_0$ and $\lambda(t_i) \rightarrow +\infty$, as $t_i\rightarrow 0$? More precisely, one can ask the following, possibly interrelated, questions. Is there a \emph{purely (local) algebro-geometric way of describing such metric bubbles}, somewhat generalizing for families Chi Li's \cite{chili} normalized volume for detecting the tangent cone? What kind of \emph{structure} does the space $\mathcal{T}=\{B_\infty\}$ of all such bubbles $B_\infty$ may have?

In this experimental elementary note we start to investigate this problem in low dimensions. This, combined with the recent differential geometric foundational work of S. Sun \cite{sun} refining Donaldson-Sun \cite{donsun2}, could give some evidence of the emergence of an algebro-geometric general theory of multiscale non-collapsing for families of KE varieties. 

\begin{remark} Of course similar questions linking algebraic and differential geometry of rescaled limits are relevant also in case of (local and global) geometric \emph{collapsing} of KE metrics. This is expected to be linked to deeper non-Archimedean aspects of a degenerating family (e.g., \cite{odaka1}), but we will not discuss this problem here.
\end{remark}

We refer the readers to the sections below for the precise results, but here we give a quick survey of the content of the paper.  In Section \ref{sec:1dim} we describe bubbling in the case of conical (log) KE metrics on Riemann surfaces, establishing a receipt for their algebraic detection (Theorem \ref{thm:flatP1}). Moreover, for the case of flat metrics on the Riemann sphere, we investigate the relation of bubbling with the Deligne-Mostow and Deligne-Mumford compactifications of moduli spaces of such flat metrics (Theorem \ref{thm:DelMum}).
In Section \ref{sec:2dim} we then move to the complex two dimensional case where we study local bubbling for the $A_k$ ALE spaces thanks to the Gibbons-Hawing ansatz description, and discuss how the metric rescaling can indeed be detected by certain algebro geometric rescaling of the deformation family (Theorem \ref{thm:aklimits}). 
The logarithmic two dimensional situation is, however, more subtle: here we describe certain examples where algebraic rescalings alone would not be able to detect the bubbles, but rather ``weighted bubbles" similar to the ``weighted tangent cone" (see \cite{sun}). Finally, in Section \ref{sec:highdim}, we put together our observations to discuss some aspects of the emerging higher dimensional picture for bubbling of non-collapsing KE metrics. \\

\emph{Acknowledgments.} We thank Song Sun for the discussions on the topic on several occasions, and for having shared with us  his paper \cite{sun}. We thank Yuji Odaka for very useful comments on a draft of this paper that improved considerably our presentation. The first author thanks Dmitri Panov for discussions on flat metrics with cone points and Deligne-Mumford compactification. 

During the writing of this paper the second author was supported by a Villum Young Investigator Grant 0019098 and Villum YIP+ 53062.

\section{Bubbling in the one dimensional logarithmic case} \label{sec:1dim}

Let \(\beta\) be a positive real number and write $\C_{\beta}$ for the 
complex plane \(\C\) endowed with the metric of a cone of total angle \(2\pi\beta\) with its vertex located at \(0\). Using a standard complex coordinate $z$ the metric is given by the line element
\begin{equation} \label{eq:conemodel}
	|z|^{\beta-1}|dz| .
\end{equation}

Here we consider constant curvature metrics on Riemann surfaces with a finite number of cone points at which the metric is asymptotic, in local complex coordinates, to the model given by Equation \eqref{eq:conemodel} above, and we want to study \emph{the blow-up limits that arise after we rescale an algebraic family of such metrics where a cluster of cone points merge into a single conical singularity}. Geometrically, multiscale bubbling happens due to the different relative speed of collision among the conical points. Note that this problem has already been considered in the literature (\cite{mazzeo, monpan1, monpan2}), but below we will further investigate certain aspects more relevant to establish an algebraic description of bubblings.

\subsection{Infinite flat metrics with cone points}
We begin with a purely local (non-compact) description of bubbling of flat conical metrics. The metric bubbles relevant to this section can be written down explicitly in complex coordinates as explained in the next Lemma \ref{lem:inflatmet} (a schematic picture of the geometry of these metric bubbles is given later in Figure \ref{fig:inflatmet}).

\begin{lemma}\label{lem:inflatmet}
Let $p_1, \ldots, p_N \in \C$ and let $\beta_1, \ldots, \beta_N \in (0,1)$ be such that 
\begin{equation}
\sum_{i=1}^N (1-\beta_i) < 1 .
\end{equation}
Then the line element
\begin{equation}\label{eq:inflatmet}
\left(\prod_{i=1}^N |z-p_i|^{\beta_i-1}\right) |dz|
\end{equation}
defines a flat K\"ahler metric on $\C$ with the following properties.
\begin{enumerate}
\item[(i)] At the points $p_i$ it is locally isometric to a cone of total angle $2\pi\beta_i$.
\item[(ii)] Outside a compact set it is isometric to the open end of a cone of total angle $2\pi\gamma$ where $\gamma \in (0,1)$ is given by
\begin{equation}\label{eq:gamma}
1-\gamma = \sum_{i=1}^N (1-\beta_i) .
\end{equation} 
\end{enumerate}
\end{lemma}

\begin{proof}
The Gaussian curvature $K$ of a conformal metric $g = e^{2u}|dz|^2$ is given by the formula $K = e^{-2u}\Delta u$, therefore $K=0$ if and only if $u$ is harmonic. In particular, the metric given by Equation \eqref{eq:inflatmet} is flat outside the points $p_i$ because the function $u = \sum_i (\beta_i-1) \log |z-p_i|$ is harmonic on $\C \setminus \{p_1, \ldots, p_N\}$.

The local uniformization theorem asserts that if $g = e^{2u}|dz|^2$  is flat then we can find a local holomorphic coordinate $w$ in which $g$ is equal to the Euclidean metric. Indeed, we can locally write the harmonic function $u$ as the real part of a holomorphic function, say  $u = \mbox{Re}(h)$. If we let $F$ be the solution of $F' = e^h$ with $F(0)=0$ and change coordinates to $w=F(z)$, then $g = |dw|^2$ as wanted. 

In order to prove items (i) and (ii) we consider a slightly more general situation.
Let \(\beta\) be a real number and let $g$ be the metric defined on a punctured disc \(D^* \subset \C\) around $0$ given by
\[g = e^{2u} |z|^{2\beta-2} |dz|^2 \]
where \(u\) is harmonic and extends smoothly over the origin.

\textbf{Claim.}  
If \(\beta \in \R \setminus \{-1, -2, \ldots\}\) then \(g\) is isometric in a neighbourhood of the origin to a cone of total angle \(2\pi\beta\) if \(\beta>0\), an infinite end of a cylinder if \(\beta=0\) or to an infinite end of a cone of total angle \(2\pi(-\beta)\) if \(\beta<0\).

We prove the claim following the arguments of Troyanov \cite[Proposition 2]{troyanov}.  Write \(u = \mbox{Re}(h)\) where \(h\) is holomorphic. We have a power series expansion  \(e^h= \sum_{k \geq 0} a_k z^k\) with \(a_0 \neq 0\) which is convergent on a disc. Assume first that \(\beta \neq 0\) and let \(f = \sum_{k\geq0} b_kz^k\) with \(b_k = (1+k/\beta)^{-1} a_k\)
(here we use that $\beta$ is not a negative integer). Then $f$ solves the equation
\begin{equation}\label{eq:solf}
\begin{aligned}
		\frac{d}{dz}\left( \frac{z^{\beta}}{\beta}f\right) &= z^{\beta-1} \sum_{k\geq 0} \left(b_k + \frac{k}{\beta}b_k\right)z^k \\
		&= z^{\beta-1} e^h .
\end{aligned}
\end{equation}
	
Note that $f(0) = a_0 \neq 0$ so we can take a single-valued holomorphic branch of \(f^{1/\beta}\) around the origin. Let
\[w= z f^{1/\beta} . \]
It follows from Equation \eqref{eq:solf} that \(w^{\beta-1}dw= z^{\beta-1}e^hdz\) and hence \(g=|w|^{2\beta-2}|dw|^2\) which implies the claim when $\beta \neq 0$.
On the other hand, if \(\beta=0\) then we solve \(df/dz = (a_0^{-1}e^h-1)/z\) and change coordinates to \(w=ze^f\) so that \(a_0w^{-1}dw = z^{-1}e^udz\) and $g=|a_0|^2 |w|^{-2}|dw|^2$. This finishes the proof of the claim.

Item (i) then follows immediately from the claim applied to $\beta = \beta_i$. Furthermore, outside a compact set $|z| \gg 1$ we can write the line element \eqref{eq:inflatmet} as $|z|^{\gamma-1} e^u |dz|$ where $u$ is harmonic and bounded. Let $\tilde{z} =1/z$ so that close to $\tilde{z}=0$ we have $|\tilde{z}|^{-\gamma-1}e^u|d\tilde{z}|$. Hence item (ii) follows from the claim applied to $\beta = -\gamma$.
\end{proof}

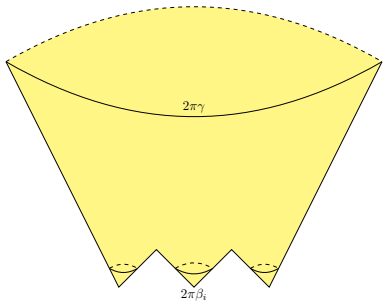
\begin{figure}[h]
	\centering
	\scalebox{.5}{
		\begin{tikzpicture}
			\fill[yellow!60] (-5, 5) -- (-2,-1) -- (-1, 0) -- (0,-1) -- (1, 0) -- (2,-1) -- (5,5) to[bend right] (-5,5);
			\draw (-5, 5) -- (-2,-1) -- (-1, 0) -- (0,-1) -- (1, 0) -- (2,-1) -- (5,5);
		  \draw[dashed] (-5,5) to[bend left] (5,5);
            \draw (-5,5) to[bend right] (5,5);
            \draw (-2.25, -.5) to[bend right] (-1.5,.-.5);
            \draw[dashed] (-2.25, -.5) to[bend left] (-1.5,.-.5);
            \draw (1.5, -.5) to[bend right] (2.25,.-.5);
            \draw[dashed] (1.5, -.5) to[bend left] (2.25,.-.5);
            \draw (-.5,-.5) to[bend right] (.5,-.5);
            \draw[dashed] (-.5,-.5) to[bend left] (.5,-.5);

            \node at (0,-1.2) {$2\pi\beta_i$};
             \node at (0,3.8) {$2\pi\gamma$};
		\end{tikzpicture}
	}
	\caption{Infinite flat metric described by Equation \eqref{eq:inflatmet}.}
	\label{fig:inflatmet}
\end{figure}

\begin{remark}
    Somewhat surprisingly, Lemma \ref{lem:inflatmet} \emph{requires} the angle at infinity $\gamma$ to be non-integer. Since we are restricting to the case that $\beta_i \in (0,1)$, it is immediate from Equation \eqref{eq:gamma} that $\gamma \in (0,1)$. On the other hand, if the angle at infinity is an integer multiple of $2\pi$ then Item (ii) of Lemma \ref{lem:inflatmet} might indeed not hold.
	For example, the infinite flat metric with two cone points of angles $\pi$ and $3\pi$ given by
	\[|z+1|^{-1/2} |z-1|^{1/2} |dz|\]
	has an end asymptotic to the Euclidean plane (i.e. $\gamma=1$) but it is not isometric to it outside a compact set. See Figure \ref{fig:integerangle}.
\end{remark}

\begin{figure}[h]
	\centering
	\scalebox{0.5}{
		\begin{tikzpicture}
			
			\fill[yellow!60] (-5, 0) -- (0,0) -- (0, -2) -- (5, -2) -- (5,3) -- (-5, 3);
			\draw (0,-2) to (0,0);
			\draw (-5,0) to (0,0);
			\draw (0,-2) to (5,-2);
			\draw[red, dashed] (-3,0) -- (-3,2) -- (3,2) -- (3,-2);
			
			\centerarc[](0,0)(-90:180:.3);
			\centerarc[](0,-2)(0:90:.3);
			
			\node[scale=1.5] at (.6, 0) {\(\frac{3}{2}\pi\)};
			\node[scale=1.5] at (.6, -1.6) {\(\frac{1}{2} \pi\)};
		\end{tikzpicture}
	}
	\caption{The double of the dashed red line gives a rectangle with a pair of unequal opposite sides. Therefore this metric can't be isometric to $\R^2$ outside a compact set.}
	\label{fig:integerangle}
\end{figure}
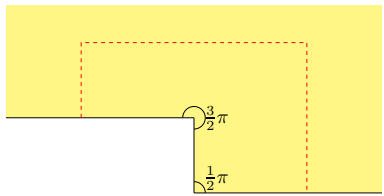

\subsection{Local models for singularity formation}

Next we consider families of infinite flat metrics $g_t$ of the form given by Equation \eqref{eq:inflatmet}, where the cone points $p_i = p_i(t)$ are \emph{holomorphic functions} on $\Delta = \{ |t| < 1\}$. We examine the case when all cone point collide at $0$, i.e. $p_i(0) = 0$ for all $i$, and describe algebraically the bubble tree $\mathcal{B}$ of all possible rescaled limits. Let's begin discussing a concrete trivial example:

\begin{example}[Collision of two cone points]
The basic case occurs when a pair of two cone points $\beta_1, \beta_2 \in (0,1)$ with $\beta_1+\beta_2>1$ collide into a single cone point $\gamma = \beta_1 + \beta_2 -1$. For instance, this is described by the family of metrics parameterized by $\epsilon(=|t|)>0$ given by
\[
g_{\epsilon} = |z+\epsilon|^{2\beta_1-2} |z -\epsilon|^{2\beta_2-2}|dz|^2 .
\]
If we let $z = \epsilon \tilde{z}$ then we see that the family
\[
g_{\epsilon} = \epsilon^{2\beta_1+2\beta_2-2} |\tz+1|^{2\beta_1-2} |\tz-1|^{2\beta_2-2}|d\tz|^2 
\]
is just a rescaling $g_{\epsilon} = \epsilon^{2\beta_1+2\beta_2-2} \cdot g_1$. The distance with respect to $g_{\epsilon}$ of the points $z = \pm \epsilon$ is a constant multiple of $\epsilon^{\beta_1+\beta_2-1}$ and the family $g_{\epsilon}$ converges to the $2$-cone $\C_{\gamma}$ of total angle $2\pi\gamma$ as $\epsilon \to 0$.
We can realize the metrics $g_{\epsilon}$ of this family by taking the \emph{double} (that is, by gluing along the boundary two copies) of the truncated wedge shown in yellow in Figure \ref{fig:collision}. As the red segment slides parallel to the left the distance between the $2$ cone points $\beta_1, \beta_2$ goes to $0$.
\end{example}

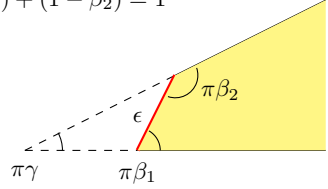
\begin{figure}[h]
	\centering
	\begin{tikzpicture}	
		\draw[dashed] (4, 0) to (5.5,0);
		\draw[] (5.5, 0) to (8,0);
		\draw[dashed] (4, 0) to (6,1);
		\draw[] (6, 1) to (8, 2);
		\draw[red, ultra thick] (5.5, 0) to (6,1);
		\fill[yellow!60] (5.5, 0) -- (8, 0) -- (8, 2) -- (6, 1) ;
		\centerarc[](5.5,0)(0:60:.3);
		\centerarc[](6,1)(-109:19:.3);
		\centerarc[](4,0)(0:28: .5);
		
		\draw  (4, -0.3) node (asd) [scale=.9] {\(\pi \gamma\)};
		\draw  (5.5, -.3) node (asd) [scale=.9] {\(\pi \beta_1\)};
		\draw  (6.6, .8) node (asd) [scale=.9] {\(\pi \beta_2\)};
        \draw (5.5,.45) node (asd) [scale=.9] {$\epsilon$};
        \draw  (4, 2) node (asd) [scale=.9] {$\gamma + (1-\beta_1) + (1-\beta_2) = 1$};

	\end{tikzpicture}
	\caption{Relation between colliding and limiting angles.}
	\label{fig:collision}
\end{figure}

In order to describe algebraically the general case of a cluster of cone points colliding we begin by reviewing a standard construction of a tree out of a finite set of holomorphic functions by grouping them according to their relative order of vanishing, as explained in \'Ettiene Ghys's book \cite[p.27]{ghys}.

Let $\mathcal{O}_{\C,0}$ be the local ring of germs of holomorphic functions defined in a neighbourhood of $0 \in \C$. For $f \in \mathcal{O}_{\C, 0}$ we let $\nu(f)$ be its order of vanishing at $0$. For any integer $k \geq 0$ we let $\sim_k$ be the equivalence relation on $\mathcal{O}_{\C,0}$ defined by $f \sim_k g$ if $\nu(f-g) \geq k$. This gives a nested sequence of equivalence relations in the sense that if $f \sim_{k'} g$ and $k' \geq k$ then $f \sim_{k} g$, i.e. the equivalence classes of $\sim_{k'}$ refine the equivalence classes of $\sim_{k}$.

Consider now a finite set $S = \{p_1, \ldots, p_N\} \subset \mathcal{O}_{\C,0}$. 
We construct a tree $\mathcal{T} = \mathcal{T}(S)$ by grouping the elements of $S$ according to their relative order of vanishing as defined next. The vertices of $\mathcal{T}$ are subsets of $S$ and the children of a node make a partition of it. We begin by defining a tree $\Tilde{\mathcal{T}}$ whose $k$-th level are the equivalence classes in $S$ given by $\sim_k$. The root of $\Tilde{\mathcal{T}}$ is the set $S$ itself, the children of $S$ are the equivalence classes of $\sim_1$ and so on. The tree $\Tilde{\mathcal{T}}$ is infinite; however, since the set $S$ is finite, there is a smallest $k_0$ such that the equivalence classes of $\sim_k$ are singletons $\{p_i\}$ for all $k \geq k_0$. 

\begin{definition}
$\mathcal{T}$ is the finite tree obtained up to the $k_0$-th level of $\Tilde{\mathcal{T}}$. Moreover, we bypass every interior node with a single child by replacing every occurrence of a triplet $\to \bullet \to$ with an edge $\to$.
\end{definition}

\begin{example}\label{ex:poltree}
If $S$ is made of the $4$ functions $p_1 = t$, $p_2 = t-t^4$, $p_3 = t + t^4$, and $p_4 = t^2$ then $\mathcal{T}$ is the tree shown in Figure \ref{fig:polytree}.
\end{example}

\begin{figure}[h]
\centering
\begin{tikzpicture}
\path (0,0) node(r) {$\{p_1, \ldots, p_4\}$} 
      (-2,-1) node(a) {$\{p_1, p_2, p_3\}$}
      (2,-1) node(b) {$\{p_4\}$}
      (-3,-2) node(c) {$\{p_1\}$}
      (-2,-2) node(d) {$\{p_2\}$}
      (-1,-2) node(e) {$\{p_3\}$};
\draw[->] (r) to (a);
\draw[->] (r) to (b);
\draw[->] (a) to (c);
\draw[->] (a) to (d);
\draw[->] (a) to (e);
\end{tikzpicture}
\caption{The tree $\mathcal{T}(S)$}
\label{fig:polytree}
\end{figure}
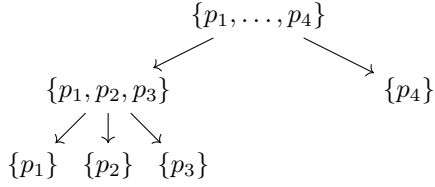

Let $S = \{p_1(t), \ldots, p_N(t)\}$ be holomorphic functions on the unit disc $\Delta$ and let $\beta_1, \ldots, \beta_N$ be real numbers in the interval $(0,1)$ such that $\sum_i (1-\beta_i) < 1$.
We proceed to define the bubble tree associated to the family of infinite flat metrics
\begin{equation}\label{eq:flatfamily}
\left(\prod_{i=1}^N |z-p_i(t)|^{\beta_i-1}\right) |dz| .
\end{equation}

Let $\textbf{v}$ be an interior (i.e. non-leaf) node of $\mathcal{T}(S)$ corresponding to a subset $V \subset \{1, \ldots, N\}$.
The children $\mathbf{v}_1, \ldots, \mathbf{v}_{\ell}$ of $\mathbf{v}$ give a partition $V_1 \cup \ldots \cup V_{\ell} = V$.
There is a non-negative integer $k$ such that all derivatives of $p_i$ at $0$ agree up to order $k$ for all $i \in V$ but the $k+1$ derivatives $p_i^{(k+1)}(0)$, $p_j^{(k+1)}(0)$ are different if $i \in V_r$ and $j \in V_s$ with $r \neq s$ and equal if $r=s$.
Let $q_1, \ldots, q_{\ell} \in \C$ be the values of the $(k+1)$-derivatives $p_i^{(k+1)}(0)$ for $i \in V$ corresponding to the $\ell$ different children of the node $\mathbf{v}$. For each $r = 1, \ldots, \ell$ we let $\Tilde{\beta}_r \in (0,1)$ be given by
\[
1-\Tilde{\beta}_r = \sum_{i \in V_r} (1-\beta_i) .
\]

\begin{definition}
For each $\textbf{v} \in \mathcal{T}$ as above we associate the following objects:
\begin{itemize}
\item $\mathcal{C}_{\mathbf{v}}$ is the $2$-cone $\C_{\tilde{\gamma}}$ of total angle $2\pi\Tilde{\gamma}$ where
\[
1-\Tilde{\gamma} = \sum_{i \in V} (1-\beta_i) .
\]
\item $B_{\mathbf{v}}$ is the infinite flat metric on $\C$ with cone angles $2\pi\Tilde{\beta}_1, \ldots, 2\pi\Tilde{\beta}_{\ell}$ at the points $q_1, \ldots, q_{\ell}$ and isometric to the end of the cone $\mathcal{C}_{\mathbf{v}}$  at infinity.
\end{itemize}
\end{definition}

\begin{definition}
The \emph{metric bubble tree} $\mathcal{B}$ is the set of all infinite flat metrics $B_{\mathbf{v}}$ labelled by the interior nodes $\mathbf{v} \in \mathcal{T}(S)$.
\end{definition}

Let $\pi: \Delta \times \C \to \Delta$ be the projection map $(t, z) \mapsto t$ and equip the fibres $\pi^{-1}(t)$ with the infinite flat metrics with cone points at $p_i = p_i(t)$ given by Equation \eqref{eq:flatfamily}. Fix a section of $\pi$ given by a holomorphic function $t \mapsto (t, s(t))$ with $s(0)=0$. 
The section determines a path in the tree \(\mathcal{T}\) by looking at the polynomials from $\{p_1, \ldots, p_N\}$ that approximate $s$ closer, as explained in the next paragraph.

Recall that $\mathcal{T} = \mathcal{T}(S)$ is the tree associated to $S= \{p_1, \ldots, p_N\}$.
Let $\mathcal{T}'$ be the tree associated to $\{s\} \cup S$. Take the path from the root of $\mathcal{T}'$ to the leaf $\{s\}$. Remove $s$ from every interior node of the path to obtain a path in $\mathcal{T}$ and label its interior nodes starting from the root of $\mathcal{T}$ as $\textbf{v}_1, \ldots, \textbf{v}_{\ell}$. See Figure \ref{fig:polytreepath}

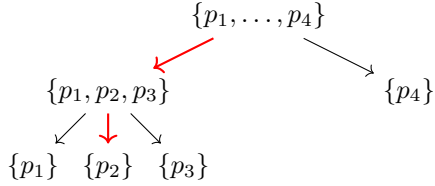
\begin{figure}[h]
\centering
\begin{tikzpicture}
\path (0,0) node(r) {$\{p_1, \ldots, p_4\}$} 
      (-2,-1) node(a) {$\{p_1, p_2, p_3\}$}
      (2,-1) node(b) {$\{p_4\}$}
      (-3,-2) node(c) {$\{p_1\}$}
      (-2,-2) node(d) {$\{p_2\}$}
      (-1,-2) node(e) {$\{p_3\}$};
\draw[->, red, thick] (r) to (a);
\draw[->] (r) to (b);
\draw[->] (a) to (c);
\draw[->, red, thick] (a) to (d);
\draw[->] (a) to (e);
\end{tikzpicture}
\caption{Path determined by the section $s(t) = t - t^4 + \text{(h.o.t.)}$ shown in red where \(\mathcal{T}\) is as in Example \ref{ex:poltree}.}
\label{fig:polytreepath}
\end{figure}

For each $\alpha>0$ we let $h_{\alpha}$ be the pointed Gromov-Hausdorff limit
\begin{equation}\label{eq:halpha}
h_{\alpha} =
\lim_{t \to 0} \left(\C, |t|^{-2\alpha} g_t, s(t) \right) .
\end{equation}
Let \(\textbf{v}_1, \ldots, \textbf{v}_{\ell}\) be the vertices of \(\mathcal{T}\) determined by the section $s$ and let \(B_{\textbf{v}_i}\) be the corresponding bubbles. The section $s$ also determines points $q_i \in B_{\textbf{v}_i}$ where $q_i$ is the value of the $k$-th derivative of the section and $k$ is the smallest integer such that the values of the  $k$-th derivatives of the functions in the equivalence class represented by \(\textbf{v}_i\) are not all equal.

\begin{lemma}\label{lem:rescaledsection}
There is an increasing sequence $0 = \alpha_0 < \alpha_1 < \ldots < \alpha_{\ell}$ such that, up to scale,
\(h_{\alpha}\) is isometric to:
	\begin{itemize}  
		\item the cone $\mathcal{C}_{\mathbf{v}_i}$ with base point the vertex if  \(\alpha_{i-1} < \alpha < \alpha_i\);
		\item the bubble $B_{\mathbf{v}_i}$ with base point $q_i$ determined by the section $s$ if \( \alpha = \alpha_i\). 
	\end{itemize}
 If $s$ is different from all $p_i$ then $h_{\alpha}$ is isometric to $\C$ for all $\alpha>\alpha_{\ell}$. While if $s = p_i$ for some $i$ then $h_{\alpha}$ is isometric to $\C_{\beta_i}$ for all $\alpha > \alpha_{\ell}$.
\end{lemma}

\begin{proof}
Center coordinates at $s(t)$ by letting
\[
\Tilde{z} = z - s(t) . 
\]
In the $\Tilde{z}$ coordinate the position of the cone points is given by 
\[
\Tilde{p}_i(t) = p_i(t) - s(t) .
\]
Write $\Tilde{p}_{j}(t) = a_{j}t^{d(j)} + \text{(h.o.t)} $
with \(a_{j} \neq 0 \) and \(d(j) \geq 1\). We have a finite set of orders of vanishing  \(D \subset \mathbb{N} \) and a partition \( \sqcup_{d \in D} I_d = \{  1, \ldots, N \} \) such that \(d(j) = d \) for all \( j \in I_d\).
Order the elements of \(D\) as \(d_1 < d_2 < \ldots <d_{\ell}\). The interior nodes $\mathbf{v}_1, \ldots, \mathbf{v}_{\ell}$ of $\mathcal{T}$ correspond to the subsets $\{\tilde{p}_j \, | \, d(j) \geq d_1 \}, \ldots, \{\tilde{p}_j \, | \, d(j) \geq d_{\ell} \}$.

If we change coordinates to $\Tilde{z} \mapsto t^{d_i}\Tilde{z}$ then the line elements of $g_t$ transform up to a positive $O(1)$ factor to
\begin{equation}\label{eq:ztilde}
|t|^{\alpha_i} \left( \prod_{j | d(j) \geq d_i} |\Tilde{z} - t^{-d_i}\Tilde{p}_j(t)|^{\beta_j -1} \right)
\left( \prod_{j | d(j) < d_i} |1 - t^{d_i}\Tilde{p}^{-1}_j(t)\Tilde{z}|^{\beta_j -1}\right) |d\Tilde{z}|     
\end{equation}
where
\[
\alpha_i = d_i \left( 1 + \sum_{j | d(j) \geq d_i} (\beta_j-1) \right) + \sum_{j | d(j) < d_i} (\beta_j-1) d_j .
\]
If we let $\gamma \in (0,1)$ be given by $\gamma - 1 = \sum_i (\beta_i -1)$ then
\begin{equation}\label{eq:alphai}
\alpha_i = d_i \gamma + \sum_{j | d(j) < d_i} (1-\beta_j) (d_i - d(j))    
\end{equation}
from which it is clear that $\alpha_1 < \ldots < \alpha_{\ell}$. It follows from Equation \eqref{eq:ztilde} that, up to scale, the metrics $h_{\alpha_i}$ are isometric to the bubbles $B_{\mathbf{v}_i}$. Indeed the terms $t^{-d_i}\Tilde{p}_j(t)$ inside the first parenthesis in Equation \eqref{eq:ztilde} converge to the $d_i$-th derivative of $\Tilde{p}_j$ at the origin 
\[
\lim_{t \to 0} \frac{\Tilde{p}_j(t)}{t^{d_i}} = \frac{\Tilde{p}_j^{(d_i)}(0)}{d_i!}
\]
while the factors $|1-t^{d_i}\Tilde{p}_j^{-1}(t) \Tilde{z}|$ in the second parenthesis in Equation \eqref{eq:ztilde} converge uniformly on compacts subsets of $\C$ to $1$ as $t \to 0$.

More generally, for \( \lambda>0 \) we dilate coordinates by \( \Tilde{z} \mapsto |t|^{\lambda} \Tilde{z} \) so that  \(\sqrt{g_t}\) is up to a positive \(O(1)\) factor
	\begin{equation}\label{eq:tildez2}
		|t|^{\alpha(\lambda)}  \left( \prod_{j | d(j) \geq \lambda}  |\Tilde{z} - |t|^{-\lambda} \Tilde{p}_j(t)|^{\beta_j -1} \right) \left( \prod_{j | d(j) < \lambda}  |1 - |t|^{\lambda} \Tilde{p}_j^{-1}(t) \Tilde{z}|^{\beta_j -1} \right) |d\Tilde{z}| ,
	\end{equation}	
where
 \[
\alpha(\lambda) = \gamma \lambda + \left( \sum_{j | d(j) < \lambda} (1-\beta_j) \right) \lambda + \sum_{j | d(j) < \lambda} (\beta_j-1)d(j) .
 \]
This defines a picewise linear continuous convex bijection \( \alpha (\lambda) : \mathbb{R}_{\geq 0} \to \mathbb{R}_{\geq 0} \) with \( \alpha (d_i) = \alpha_i \). If $\alpha_{i-1} < \alpha < \alpha_i$ then $d_{i-1} < \lambda < d_i$ therefore all the terms $|t|^{-\lambda}\Tilde{p}_j(t)$ inside the first parenthesis in Equation \eqref{eq:tildez2} converge to $0$ as $t \to 0$; and the line elements given by Equation \eqref{eq:tildez2} converge to the cone $\mathcal{C}_{\mathbf{v}_i}$ as $t \to 0$.
\end{proof}

\subsection{Bubbling for flat metrics on the $2$-sphere}

We recall the well-known classification of flat metrics with conical singularities on the $2$-sphere $S^2$ in terms of their cone angles and conformal structures. We follow \cite{thurston, troyanov}.

Let $g$ be a flat metric on $S^2$ with a finite number of cone points $p_1, \ldots, p_N$ of total angle $2\pi\beta_1, \ldots, 2\pi\beta_N$ where $\beta_i \in (0,1)$. In precise terms, $g$ is a smooth flat metric on $S^2 \setminus \{p_1, \ldots, p_N\}$ and in a neighbourhood of each $p_i$ the metric $g$ is equivalent in geodesic polar coordinates centred at $p_i$ to the $2$-cone of total angle $2\pi\beta_i$ given by
\(dr^2 + \beta_i^2r^2d\theta^2 \).
More geometrically, by the Alexandrov embedding theorem \cite[Theorem 37.1]{pak} the metric $g$ is either isometric to the surface of a convex polyhedron in $\R^3$ or the double of a convex polygon in $\R^2$.

By Gauss-Bonnet the cone angles satisfy
\begin{equation}\label{eq:GB}
    \sum_{i=1}^N (1-\beta_i) = 2 .
\end{equation}
Conversely we have the following.

\begin{lemma}\label{lem:fms1}
Let $\beta_1, \ldots, \beta_N \in (0,1)$ such that the Gauss-Bonnet constraint \eqref{eq:GB} is satisfied and let $p_1, \ldots, p_N$ be distinct points in the complex plane $\C$. Then the line element
\begin{equation}\label{eq:flatsphere}
    \left( \prod_i |z - p_i|^{\beta_i-1} \right) |dz|
\end{equation}
extends smoothly over infinity to define a metric $g$ on the $2$-sphere with $N$ cone points of angles $2\pi\beta_1, \ldots, 2\pi\beta_N$. 
\end{lemma}

\begin{proof}
An easy check shows that if we let $z=1/w$ then the line element \eqref{eq:flatsphere} takes the form
\[
\left( \prod_i |1 - p_i w|^{\beta_i-1} \right) |dw|
\]
which extends smoothly over $w=0$. On the other hand, the proof of Lemma \ref{lem:inflatmet} shows that close to every $p_i$ we can find a holomorphic coordinate $\xi$ in which the metric $g$ agrees with the flat $2$-cone $|\xi|^{\beta_i-1}|d\xi|$.
\end{proof}

\begin{lemma}\label{lem:fms2}
Every flat metric $g$ on the $2$-sphere with $N$ cone point of total angles $2\pi\beta_1, \ldots, 2\pi\beta_N$ is isometric to a metric of the form given by Equation \eqref{eq:flatsphere}.
\end{lemma}

\begin{proof}
The metric $g$ induces in the usual way a complex structure away from the cone points given by an anti-clockwise rotation of $90^o$. This complex structure extends smoothly over the cone points as can be checked for the $2$-cone. Indeed, if $\xi = r^{1/\beta} e^{i\theta}$ then $dr^2 + \beta^2r^2d\theta^2 = \beta^2 |\xi|^{2\beta-2}|d\xi|^2$ showing that the $2$-cone endows the Euclidean plane $\R^2$ with the complex structure of $\C$. In particular, the metric $g$ endows the $2$-sphere with the structure of a Riemann surface. By the Uniformization Theorem we know that this Riemann surface is biholomorphic to $\CP^1$. Let $z$ be a complex coordinate on $\C \subset \CP^1$ such that $\infty$ is a smooth point of $g$ and let $p_1, \ldots, p_N \in \C$ be the cone points. Then we can write
\[
g = e^{2u} \left( \prod_i |z - p_i|^{2\beta_i-2} \right) |dz|^2
\]
for a real function $u$. Since $g$ is flat then $u$ is harmonic away from the cone points. The fact that the cone angle at $p_i$ is $2\pi\beta_i$ implies that $u$ extends smoothly over $p_i$. As a consequence, the function $u$ is smooth and harmonic on the whole $\CP^1$ and by the maximum principle it is constant.
\end{proof}

Now, let $\mathcal{M}_{0,N}$ be the moduli space of configurations of $N$ marked points in the Riemann sphere $\CP^1$ modulo the action of linear transformations $PSL(2, \C)$ that preserve the markings. Let $Met_{\beta}$ be the space of all flat metrics on the $2$-sphere with $N$ cone points of angles $2\pi\beta_1, \ldots, 2\pi\beta_N$ modulo marked isometries and scale.
 
 \begin{corollary}\label{lem:forget}
 The forgetful map
 \begin{equation}\label{eq:forget}
 Met_{\beta} \to \mathcal{M}_{0,N}
 \end{equation}
 that records the conformal structure induced by the metric is a bijection.
 \end{corollary}
 
\begin{proof}
Surjectivity follows from Lemma \ref{lem:fms1} and injectivity  from Lemma \ref{lem:fms2}.
\end{proof}

Consider a sequence of unit area flat conical metrics $g_t$ on $\CP^1$ parameterized by $t$ in the unit disc $\Delta$. 
We assume that the positions of the cone points $p_1(t), \ldots, p_N(t)$ depend holomorphically on $t \in \Delta$ and that for $t \neq 0$ we have $N$ different cone points ($p_i(t) \neq p_j(t)$ for $i \neq j$) of fixed cone angles $2\pi\beta_1, \ldots, 2\pi\beta_N$. However, for $t=0$ we allow some of the cone points come together. More precisely, we have a partition of the index set
\[
\bigcup_{k=1}^M I_k = \{1, \ldots, N\}
\]
by disjoint subsets $I_k$ such that $p_i(0) = p_j(0)$ if $i, j$ belong to the same subset $I_k$ and $p_i(0) \neq p_j(0)$ if $i \in I_k$, $j \in I_{\ell}$ with $k \neq \ell$. Moreover, we assume that for every $k=1, \ldots, M$ we have
\begin{equation}\label{eq:collisions}
    \sum_{i \in I_k} (1-\beta_i) < 1 .
\end{equation}
Let $q_1, \ldots, q_M$ be the values of $p_1(0), \ldots, p_N(0)$ so that the limiting metric is 
\begin{equation}\label{eq:g0}
    g_0 = C_0 \left( \prod_{j=1}^M |z - q_j|^{2\gamma_j - 2} \right) |dz|^2
\end{equation}
where the angles $\gamma_j$ are determined by $1 - \gamma_j = \sum_{i \in I_j}(1-\beta_i)$ and $C_0>0$ is determined by unit area normalization.

\begin{theorem}\label{thm:flatP1}
Let $\sigma$ be a section of the projection map $(t, z) \mapsto t$ from $\Delta \times \CP^1$ to $\Delta$ given by $t \mapsto \sigma(t) = (t, s(t))$ where $s$ is a holomorphic function with $s(0) = q_j$ for some $1 \leq j \leq M$. Then the rescaled pointed Gromov-Hausdorff limits
\[
h_{\alpha} = \lim_{t \to 0} \left(\CP^1, |t|^{-2\alpha} \cdot g_t, s(t) \right)
\]
for $\alpha>0$ are given as in Lemma \ref{lem:rescaledsection} with $\mathcal{T} = \mathcal{T}(S)$ where $S = \{p_i \, | \, i \in I_j\}$.
\end{theorem}

\begin{proof}
The metrics $g_t$ are given for $t \neq 0$ by the explicit formula
\begin{equation}\label{eq:gts}
g_t = C_t \left( \prod_{i=1}^N |z - p_i(t)|^{2\beta_i-2} \right) |dz|^2     
\end{equation}
where $C_t > 0$ are determined by the condition $\text{Area}(g_t) = 1$. At $t=0$ the limiting metric $g_0$ is given by Equation \eqref{eq:g0}. It is easy to check that the distance functions induced by the metrics $g_t$ converge uniformly to the distance function given by $g_0$ and that there is a uniform $C>0$ such that $C^{-1} < C_t < C$ for all $t$ with $|t|< 1/2$. Given this, the same proof of Lemma \ref{lem:rescaledsection} applies to this case.
\end{proof}

\begin{remark}
    Using gluing methods we expect Theorem \ref{thm:flatP1} to hold for non-collapsed sequences of hyperbolic and spherical metrics. This goes  by producing approximate solutions, using the infinite flat model families, and then perturbing the approximation by analyzing the linearization of the singular Liouville equation. See \cite{mazzeo}.
\end{remark}

\subsection{Relations to Deligne-Mostow and Deligne-Mumford compactifications}
The space $\mathcal{M}_{0,N}$ of configurations of $N \geq 3$ marked points in the Riemann sphere is an open complex manifold of dimension $N-3$. The automorphisms group of the projective line $PSL(2, \C)$ acts diagonally on the Cartesian product $(\CP^1)^N$.
This action is free away from diagonals. If we let
\[
U = (\CP^1)^N \setminus \bigcup_{i \neq j} \{z_i=z_j\}
\]
then $\mathcal{M}_{0, N}$ is the space of $PSL(2, \C)$-orbits on the open set $U$.

The \emph{Deligne-Mostow compactification} $\overline{\mathcal{M}^{\beta}_{0, N}}$ of the configuration space $\mathcal{M}_{0,N}$ is defined as follows.
Fix $0<\beta_i < 1$ for $i = 1, \ldots, N$ such that the Gauss-Bonnet constraint
\[
\sum_{i=1}^N (1-\beta_i) = 2
\]
is satisfied. Moreover, we assume the generic condition that\footnote{In terms of metrics Equation \eqref{eq:noncollapse} rules out collapsing sequences.} 
\begin{equation}\label{eq:noncollapse}
    \sum_{i \in I} (1-\beta_i) \neq 1  \, \text{ for all subsets } I \subset \{1, \ldots, N\} .
\end{equation}

Let $z = (z_1, \ldots, z_N)$ be a point in $(\CP^1)^N$. The point $z$ defines a partition of the index set $\{1, \ldots, N\}$ by subsets $I_1, \ldots, I_{\ell}$ such that $z_i=z_j$ if and only if both $i, j$ belong to the same subset $I_k$. We say that $z$ is $\beta$-\emph{stable} if
\begin{equation}\label{eq:stability}
    \sum_{i \in I_j} (1-\beta_i) < 1 \, \text{ for all } 1 \leq j \leq \ell .
\end{equation}

In particular, if $z \in U$ then the subsets $I_j$ are singletons
and the stability condition \eqref{eq:stability} trivially holds. However, for arbitrary tuples in $(\CP^1)^N$ the notion of stability depends on the angle parameters $\beta = (\beta_1, \ldots, \beta_N)$. The upshot is that the space of $PSL(2, \C)$-orbits of $\beta$-stable points is a projective manifold $\overline{\mathcal{M}}^{\beta}_{0,N}$ that contains $\mathcal{M}_{0,N}$ as a Zariski open subset, see \cite[Section 4]{DM}. Under the correspondence between $\mathcal{M}_{0, N}$ and the space of unit area flat metrics on the $2$-sphere with prescribed cone angles $Met_{\beta}$ given by Corollary \ref{lem:forget}, the projective manifold $\overline{\mathcal{M}}^{\beta}_{0,N}$ corresponds to the Gromov-Hausdorff compactification of $Met_{\beta}$.

\begin{example}\label{ex:delmost}
Take $N=5$ and suppose that the first $4$ cone points are equal $\beta_1 = \ldots = \beta_4 = \beta$ while $\beta_5$ is determined by Gauss-Bonnet
\[
\beta_5 = 3 - 4\beta .
\]
The requirement that $0< \beta_5 < 1$ implies that $1/2 < \beta < 3/4$. For $2/3 < \beta < 3/4$ the space $\overline{\mathcal{M}}^{\beta}_{0,5}$ is equivalent to the projective plane $\CP^2$ and the boundary divisor (i.e. the complement of $\mathcal{M}_{0, 5}$) is the arrangement of $6 = \binom{4}{2}$ lines that join $2$ out of $4$ points in general position. When $1/2 < \beta < 2/3$ the space $\overline{\mathcal{M}}^{\beta}_{0,5}$ is equivalent to the blow-up $\text{Bl}_4\CP^2$ of the projective plane at the $4$ triple points of the arrangement. The boundary divisor is made of $10 = \binom{5}{2}$ rational curves that intersect in a normal crossing configuration. Equivalently, $\mathcal{M}_{0,5}$ is the complement in $\CP^1 \times \CP^1$ of the $7$ lines: $6$ where one of the coordinates is $0, 1, \infty$ plus the diagonal, and  $\overline{\mathcal{M}}^{\beta}_{0,5}$ for $1/2<\beta<3/4$ is obtained by blowing up the points $(0,0), (1,1), (\infty, \infty) \in \CP^1 \times \CP^1$.
\end{example}

\begin{remark}
As Example \ref{ex:delmost} shows, the boundary divisor $\overline{\mathcal{M}}^{\beta}_{0,N} \setminus \mathcal{M}_{0,N}$ is not always normal crossing.    
\end{remark}

The \emph{Deligne-Mumford} compactification $\overline{\mathcal{M}}_{0,N}$ is a projective manifold that contains $\mathcal{M}_{0,N}$ as a Zariski open subset and such that its complement $\overline{\mathcal{M}}_{0,N} \setminus \mathcal{M}_{0,N}$ is a normal crossing divisor
\[
D = \sum_P D_P
\]
whose irreducible components $D_P$ are in one to one correspondence with partitions $P=\{I_0, I_1\}$ of in the index set $\{1, \ldots, N\}$ into $2$ subsets $I_0, I_1$ such that $\min \{|I_0|, |I_1|\} \geq 2$.

The boundary points $m \in \overline{\mathcal{M}}_{0,N}$  represent connected nodal curves $C_m$ with $N$ marked points whose irreducible components are $\CP^1$'s. The total number of nodal and marked points in each irreducible component is at least $3$. The topology of $C_m$ is encoded by a tree whose vertices represent the irreducible components of $C_m$ and and an edge connects $2$ vertices if the corresponding $\CP^1$'s intersect at a nodal point. The number of divisors that contain $m$ is equal to the number of nodes of $C_m$. More precisely, splitting a node $y \in C_m$ into $2$ divides the curve $C_m$ into $2$ connected components $C_0, C_1$ and then we have a partition $P = \{I_0, I_1\}$ by recording which indices correspond to marked points in each component.

We proceed to relate the Deligne-Mostow and Deligne-Mumford compactifications.
Same as before, fix $0< \beta_i < 1$ for $i=1, \ldots, N$ such that the Gauss-Bonnet constraint $\sum_i (1-\beta_i) =2$ is satisfied. 

\begin{lemma}[\cite{koziarz}]
    There is a logarithmic resolution
    \begin{equation}
    \pi: \overline{\mathcal{M}}_{0,N} \to \overline{\mathcal{M}}^{\beta}_{0,N} .
\end{equation}
\end{lemma}

\begin{proof}
    We follow \cite[Section 5]{koziarz}. For simplicity of notation, we write $b(x_i) = 1-\beta_i$ for $i=1, \ldots, N$ and refer to it as the weight at $x_i$. The Gauss-Bonnet formula is equivalent to $\sum_i b(x_i) = 2$.
    
    Let $m \in \overline{\mathcal{M}}_{0,N}$ and let $C_m$ be the corresponding nodal curve with marked points $x_1, \ldots, x_N$. Split each node of $C_m$ into a pair of points $\{y, y'\}$ and let $Y, Y'$ be the $2$ connected components of $C_m$ that contain $y, y'$ respectively. Define
    \begin{equation}\label{eq:yweight}
    b(y) = \sum_{x_i \in Y'} b(x_i) .    
    \end{equation}
    In particular $b(y) + b(y') = 2$ and exactly one of the alternatives $b(y) < 1, b(y') > 1$ or $b(y)> 1, b(y') < 1$ must hold.
    
    Given an irreducible component $C_j$ of $C_m$ we have a finite set of marked points
    \[
    \Sigma_j = \{x_i \, | \, x_i \in C_j\} \bigcup \{y_i \, | \, y_i \in C_j\} 
    \]
    and the sum of weights over all points in $\Sigma_j$ is equal to $2$. The key elementary fact \cite[Lemma 5.1]{koziarz} is that there is a unique irreducible component $C_j$ such that $b(z) < 1$ for all $z \in \Sigma_j$, called the $\beta$-\emph{principal} component of $C_m$.

    The $\beta$-principal component $C_j$ gives us an $N$-tuple $(z_1, \ldots, z_N)$ where $z_i = x_i$ if $x_i \in C_j$ and $z_i = y$ if $x_i \in Y'$ where $y$ is a nodal point in $C_j$. By definition, the $N$-tuple $(z_1, \ldots, z_N)$ is $\beta$-stable and the resolution is given by the map
    \[
    \pi(m) = [(z_1, \ldots, z_N)]  . \qedhere
    \]
\end{proof}

\begin{example}
    Take $N=5$ and $\beta_1 = \ldots = \beta_4 = \beta$ with $1/2 < \beta < 3/4$ as in Example \ref{ex:delmost}. Let $\pi: \overline{\mathcal{M}}_{0,5} \to \overline{\mathcal{M}}_{0,5}^{\beta}$. If $1/2 < \beta < 2/3$ then $\pi$ is an isomorphism.
    If $2/3 < \beta < 3/4$ then $\pi$ is the blow-up at the $4$ triple points of the boundary divisor.
\end{example}

We can now provide a metric interpretation, in  relation to bubbling, of the boundary points in the Deligne-Mumford compactification.


\begin{theorem}\label{thm:DelMum}
Every point in the Deligne-Mumford compactification represents a bubble tree for a family of flat metrics with cone points on $\mathbb{CP}^1$ as in Theorem \ref{thm:flatP1}. Conversely, every such family of metrics determines a unique point in the Deligne-Mumford compactification $ \overline{\mathcal{M}}_{0,N}$
\end{theorem}

\begin{proof}
Let $m \in \overline{\mathcal{M}}_{0,N}$ and let $F: \Delta \to \overline{\mathcal{M}}_{0,N}$ be a holomorphic map with $F(0)=m$ and $F(\Delta^*) \subset \mathcal{M}_{0,N}$. Let $C_m$ be the nodal curve with marked points $x_1, \ldots, x_N$ corresponding to $m$. The irreducible components of $C_m$ make a tree and we set the root of this tree to be the $\beta$-principal component. 
The map $F$ gives us a family of flat metrics (Corollary \ref{lem:forget}) for which
the $\beta$-principal component of $C_m$ represents the limiting compact metric in $\CP^1$. Every other component $C_j$ has a unique nodal point $y$ with $b(y) > 1$ and represents an infinite flat metric with a cone end of total angle $2\pi(b(y)-1)$ and cone points of angles $2\pi(1-b(x_i))$ at the $x_i \in C_j$ and $2\pi(1-b(y_i))$ at all the other nodes $y_i \in C_j$ which are different from $y$. The children $q_1, \ldots, q_k$ of the $\beta$-principal component represent clusters of cone points that coalesce to single cone points in the limiting compact metric. The sub-trees which have as roots $q_1, \ldots, q_k$ represent the metric bubble trees that arise after taking rescaled limits as in Theorem \ref{thm:flatP1}. Conversely, given a family of metrics as in Theorem \ref{thm:flatP1} we define a nodal curve $C_m$ with marked points $x_1, \ldots, x_N$ whose irreducible components are the vertices of the metric bubble tree, and two such curves meet at a nodal point $y_j$ if and only if their corresponding vertices are connected. The marked points $x_i$ are the limiting (non-clustered) cone points of the family while the nodal points $y_i$ represent a collision of a cluster of cone points for one component and a cone end for the other.   
\end{proof}

\begin{remark} The Deligne-Mumford compactification carries also the differential geometric meaning as moduli compactification of \emph{hyperbolic metrics} with cusps at the marked points, with further cusps' formation at the nodal points. Thus the above theorem provides a different differential geometric interpretation of the same moduli space. It may be interesting to study further if the Hassett's moduli compactifications \cite{hassett} provide also this combined meaning of moduli of conical hyperbolic metrics (with their degenerations) and bubbles (up to certain scale). See also the discussion in Section \ref{mK}. 
\end{remark}

\section{Bubbling in two dimensions}\label{sec:2dim}

It is well-known that the singularities forming in non-collapsing sequences of Kähler-Einstein manifolds of dimension $2$ are isolated orbifold singularities \cite{Anderson, BKN} bubbling ALE spaces \cite{Bando}. Here, in analogy to what we described in the previous section, we investigate relations of bubbling and algebraic geometry for the simplest type of such singularities, namely $A_k$-singularities, showing that the picture is essentially analogous to the one dimensional log case. However, we also point out (section \ref{Log2}) that if one instead considers the more general case of log KE  metrics (so conical along a divisor) in this dimension, the bubbling picture seems to be more complicated, and related to the jumping phenomena as recently pointed out in \cite{sun}.

\subsection{The $A_k$-singularity}

Let $k$ be a positive integer and consider the cyclic group  $\Gamma_{k+1} \subset SU(2)$ of order $k+1$ generated by the diagonal matrix with eigenvalues $\exp(2\pi i / (k+1))$ and $\exp(-2\pi i/ (k+1))$. The polynomial functions $u=X_1^{k+1}$, $v = X_2^{k+1}$, and $z = X_1 X_2$ are invariant under the action of $\Gamma_{k+1}$ on $\C^2$ and give an isomorphism between the orbifold quotient and the $A_k$-singularity
\begin{equation}
    \C^2/ \Gamma_{k+1} \cong \{ uv = z^{k+1} \} \subset \C^3 .
\end{equation}

The group $\Gamma_{k+1}$ preserves the Euclidean metric on $\C^2$ and it acts freely on the unit $3$-sphere. Thus the $A_k$-singularity comes equipped with a flat K\"ahler cone metric $d\rho^2 + \rho^2 g_{S^3/\Gamma_{k+1}}$ where
\[
\rho^2 = |u|^{\frac{2}{k+1}} + |v|^{\frac{2}{k+1}}
\]
measures the intrinsic squared distance to the vertex located at $0$. The linear $\C^*$-action on $\C^3$ with weights $(k+1, k+1, 2)$ given by
\begin{equation}\label{eq:akweights}
    \lambda \cdot (u, v, z) = (\lambda^{k+1} u, \lambda^{k+1} v , \lambda^{2}z)
\end{equation}
preserves the $A_k$-singularity and scales the intrinsic distance by $|\lambda|$.

\subsection{Gibbons-Hawking ansatz}
We recall the Gibbons-Hawking construction of ALE manifolds of type $A_k$.
Let $x_1, \ldots, x_{k+1}$ be distinct points in $\R^3$ and let $f$ be the harmonic function
\begin{equation}
    f(x) = \frac{1}{2} \sum_{i=1}^{k+1} \frac{1}{|x-x_i|} .
\end{equation}

Let $\pi_0: M_0 \to \R^3 \setminus \{x_1, \ldots, x_{k+1}\}$ be the circle bundle with first Chern class $-1$ at spheres around the punctures equipped with a connection $\eta \in \Omega^1(M_0)$ such that 
\begin{equation}
    d\eta = - *df .
\end{equation}
Then 
\begin{equation}
    g = f g_{\R^3} + f^{-1} \eta^2
\end{equation}
defines a complete hyperk\"ahler metric on a $4$-manifold $M = M_0 \cup \{\Tilde{x}_i\}$ obtained by adding $k+1$ points $\Tilde{x}_1, \ldots, \Tilde{x}_{k+1}$ to $M_0$. The metric $g$ is asymptotic at infinity to the flat orbifold $\C^2/\Gamma_{k+1}$ where $\Gamma_{k+1} \subset SU(2)$ is the cyclic group generated by the diagonal matrix with eigenvalues $\exp(2\pi i / (k+1))$ and $\exp(-2\pi i/ (k+1))$. 
The manifold $M$ admits a circle action which preserves the hyperk\"ahler structure with moment map $\pi: M \to \R^3$.
The circle action has $k+1$ fixed points at $\Tilde{x}_1, \ldots, \Tilde{x}_{k+1}$ and $\pi(\Tilde{x}_i) = x_i$, while the restriction of $\pi$ to $M_0$ is equal to $\pi_0$. 

Every vector $v$ in the unit sphere $S^2 \subset \R^3$ determines a parallel complex structure $I_v$ that sends the horizontal lift of the constant vector field $v$ to the derivative of the circle action.
Let $x^1, x^2, x^3$ be linear coordinates in $\R^3$. Consider the complex structure $I$ determined by the $\p/\p x_3$ vector field. Then $z = x^1 + i x^2$ is a circle invariant holomorphic function on $(M,I)$. As shown in \cite{lebrun}, one can further produce holomorphic functions $u,v$ on $(M,I)$ of weights $1,-1$ for the circle action which give an equivariant biholomorphism between $(M, I)$  and the complex surface in $\C^3$ defined by the equation
\[
uv = \prod_{i=1}^{k+1} (z - z_i) 
\]
where $z_i = z(x_i)$, equipped with the circle action $e^{it} \cdot (u, v, z) = (e^{it}u, e^{-it}v, z)$.

If $s \subset \R^3$ is a segment that connects two points $x_i, x_j$ then the preimage $\pi^{-1}(s)$ is a $2$-sphere $S_{ij} \subset M$. The second homology group $H_2(M, \mathbb{Z})$ is generated by such spheres. If $v$ is a unit vector in $\R^3$ corresponding to a parallel complex structure $I_v$ on $M$ then the cohomology class of the K\"ahler form $\omega_v = g (I_v \cdot, \cdot)$ is determined by its pairing with the $2$-spheres $S_{ij}$ given by
\begin{equation}
    \frac{1}{2\pi} \int_{S_{ij}} \omega_v = \langle v, x_i - x_j \rangle 
\end{equation}
where $\langle \cdot, \cdot \rangle$ denotes the Euclidean inner product in $\R^3$. In particular, if $v = \p / \p x^3$ and the cohomology class of the K\"ahler form $\omega = \omega_v$ vanishes then we can assume that all the points $x_i$ lie on the plane $x^3=0$ which we identify with $\C$ via $z= x^1 + i x^2$.

\begin{remark}
All the above can be extended to ALE orbifolds. In this setting the points $x_i$ have multiplicities $m_i \in \mathbb{Z}_{\geq 1}$. The harmonic potential is
\[f(x) = \frac{1}{2} \sum_i \frac{m_i}{|x-x_i|}\]
and the metric has orbifold singularities of type $A_{m_i-1}$ at the points with $m_i > 1$. The asymptotic cone at infinity is $\C^2/\Gamma_{k+1}$ where $k+1 = \sum_i m_i$.
\end{remark}

\subsection{Local bubbling models for $A_k$-singularities}
We provide model families of ALE manifolds of type $A_k$. Let $z_i(t)$ for $i=1, \ldots, k+1$ be holomorphic functions of $t$, for $t$ in the unit disc $\Delta$ 
with $z_i(t) \neq z_j (t)$ for $i\neq j$ and $t \in \Delta^*$ and $z_i(0) = 0$ for all $1 \leq i \leq k+1$.
We consider the family of Gibbons-Hawking metrics $g_t$ given by the harmonic potentials
\[
f_t(x) = \frac{1}{2} \sum_i \frac{1}{|x - x_i(t)|}
\]
where $x_i(t) = (z_i(t), 0)$ under the identification of $\R^3$ with $\C \times \R$ given by $(x^1, x^2, x^3) \mapsto (z, x^3)$ where $z = x^1 + i x^2$. We take the complex structure given by the $x_3$-axis, so the corresponding complex surfaces are
\begin{equation}
  X_t = \{uv = \prod_{i=1}^{k+1} (z - z_i(t)) \} \subset \C^3
\end{equation}
and the K\"ahler forms $\omega_t$ are all $\p\Bar{\p}$-exact by our choice of points $x_i(t)$.

Let $\mathcal{X}$ be the complex $3$-fold which is the total space of the family. So we have a holomorphic submersion
\[
\Pi : \mathcal{X} \to \Delta
\]
whose fibers are the complex surfaces $X_t$.
Same as in Section \ref{sec:1dim} we can construct a tree $\mathcal{T} = \mathcal{T}(S)$ from the set of holomorphic functions $S = \{z_1(t), \ldots, z_{k+1}(t)\}$ by grouping them according to their relative order of vanishing at the origin.

\begin{theorem}\label{thm:aklimits}
    The set of \emph{non-cone} pointed Gromov-Hausdorff limits
    \begin{equation}\label{eq:ptGHGH}
        h_{\alpha} = \lim_{t \to 0} |t|^{-2\alpha} \cdot (X_t, g_t, \sigma(t))
    \end{equation}
    for $\alpha>0$ and $\sigma$ a holomorphic section of $\Pi$ is in one-to-one correspondence with the set of interior (i.e. non-leaf) vertices of $\mathcal{T}$. Each vertex $\mathbf{v} \in \mathcal{T}$ corresponds to an orbifold ALE space $B_{\mathbf{v}}$ asymptotic to $\C^2/\Gamma_{\ell+1}$ where $\ell+1$ is the number $|\mathbf{v}|$ of functions $z_i(t)$ in the equivalence class represented by $\mathbf{v}$. If $\mathbf{w}$ is a child of $\mathbf{v}$ then it corresponds to an orbifold point of $B_{\mathbf{v}}$ of type $A_{|\mathbf{w}|}$.
\end{theorem}

\begin{proof}
    Let $\lambda>0$ and let $(M, g)$ and $(\Tilde{M}, \Tilde{g})$ be Gibbons-Hawking ALE spaces determined by the monopole points $x_1, \ldots, x_{k+1} \in \R^3$ and $\Tilde{x}_1, \ldots, \Tilde{x}_{k+1} \in \R^3$ where $\Tilde{x}_i = \lambda x_i$ for all $i$. Then the metric $\Tilde{g}$ is isometric to $\lambda g$ by isometries that act transitively on the circle fibres over $0 \in \R^3$. Indeed, if $m_{\lambda}$ denotes the scalar multiplication by $\lambda$ in $\R^3$ then the respective harmonic potentials of the metrics $g, \Tilde{g}$ satisfy
    \[
    \Tilde{f} \circ m_{\lambda} = \lambda^{-1} f .
    \]
    We can lift $m_{\lambda}$ to a circle bundle map $\Phi: M \to \Tilde{M}$ that preserves the respective connections, i.e. $\Phi^* \Tilde{\eta} = \eta$. Pulling back by $\Phi$ we see that
    \begin{equation}\label{eq:ghscale}
        \lambda g = \lambda \left( f g_{\R^3} + f^{-1} \eta^2 \right) = \Phi^* \left( \Tilde{f} g_{\R^3} + \Tilde{f}^{-1} \Tilde{\eta}^2 \right) = \Phi^* \Tilde{g} .
    \end{equation}
    
    Fix a section $\sigma = (u(t), v(t), z(t))$. To prove the lemma it suffices to show that the pointed limits \eqref{eq:ptGHGH} for $\alpha>0$ behave in the same manner as  described in Lemma \ref{lem:rescaledsection}.
    Compose $\sigma$  with the moments maps $\pi_t: X_t \to \R^3$ to obtain $c(t) = (z(t), x^3(t))$ with $c(0)=0$. For each $t$ we change the moments map $\pi_t$ by subtracting the constant $c(t)$. This way we can assume that $c(t)$ is identically zero and the position of the monopole points $x_i(t)$ changes to $\Tilde{x}_i(t) = (\Tilde{z}_i(t), -x^3(t))$ where $\Tilde{z}_i(t) = z_i(t) - z(t)$.
    Let $d_1, \ldots, d_{\ell}$ be the order of vanishing at $t=0$ of $\Tilde{z}_1(t), \ldots, \Tilde{z}_{k+1}(t)$.
    To prevent collision of points with order of vanishing $d_i$ we change $x = (z, x^3)$ to $\Tilde{x} = (\Tilde{z}, x^3)$ where $z = t^{d_i} \Tilde{z}$. Taking the limit as $t \to 0$ all points $\Tilde{z}_j(t)$ whose order of vanishing at $t=0$ is $> d_i$ collide at $0$, all points with order of vanishing is $< d_i$ are sent off to infinity, and all the points with order of vanishing equal to $d_i$ converge to a limiting configuration given by taking the $d_i$-derivatives at $0$. It follows from Equation \eqref{eq:ghscale} that the pointed Gromov-Hausdorff limit \eqref{eq:ptGHGH} for $\alpha = d_i / 2$ is the Gibbons-Hawking ALE orbifold pointed at an $A_{|\mathbf{v}|-1}$ orbifold point where $\mathbf{v} \in \mathcal{T}$ represents all points that equal the section $z(t)$ to order $>d_i$. The configuration of monopole points is $0$ with multiplicity $|\mathbf{v}|$ and then one monopole point for each child $\mathbf{w}$ of $\mathbf{v}$ with multiplicity $|\mathbf{w}|$. In the same way, using Equation \eqref{eq:ghscale}, we see that the pointed Gromov-Hausdorff limit \eqref{eq:ptGHGH} for $d_i/2 < \alpha <d_{i+1}/2$ is the cone $\C^2/\Gamma_{|\mathbf{v}|}$ pointed at its vertex.
\end{proof}

\begin{remark}\label{rmk:ansatz}
    Given the section $\sigma(t) = (u(t), v(t), z(t))$ we can change coordinates to $u = \Tilde{u} + u(t)$, $v = \Tilde{v} + v(t)$, and $z = \Tilde{z} +z(t)$. In $\Tilde{u}, \Tilde{v}, \Tilde{z}$ coordinates the family is
    \begin{equation}\label{eq:akfamily}
        \Tilde{u} \Tilde{v} = \left( \prod_i (\Tilde{z} - \Tilde{z}_i(t)) \right) - \ell_t(\Tilde{u}, \Tilde{v})
    \end{equation}
    where $\ell_t = v(t)\Tilde{u} + u(t)\Tilde{v} + u(t)v(t)$,  $\Tilde{z}_i(t) = z_i(t) - z(t)$, and the section $\sigma$ is identically $0$. We rescale coordinates $\Tilde{u}, \Tilde{v}, \Tilde{z}$ according to the weights of the cone $\C^2/ \Gamma_{k+1}$ given by Equation \eqref{eq:akweights} as follows
    \begin{equation}
        \Tilde{u} \mapsto t^{(k+1)\lambda} \Tilde{u}, \hspace{2mm} \Tilde{v} \mapsto t^{(k+1)\lambda} \Tilde{v}, \hspace{2mm} \Tilde{z} \mapsto t^{2\lambda} \Tilde{z} .
    \end{equation}
    where $\lambda>0$ is to be determined later. Let $d \geq 1$ be the lowest order of vanishing of the functions $\Tilde{z}_i(t)$ at $t=0$. Take $\lambda = d/2$ then in the rescaled coordinates the family \eqref{eq:akfamily} is
    \begin{equation}\label{eq:akansatz}
        \Tilde{u} \Tilde{v} = \left( \prod_i (\Tilde{z} - t^{-d}\Tilde{z}_i) \right) - \Tilde{\ell}_t
    \end{equation}
    where $\Tilde{\ell}_t \to 0$ as $t \to 0$. Taking the limit as $t \to 0$ in Equation \eqref{eq:akansatz} we recover the minimal bubble limit as in Theorem \ref{thm:aklimits}.
\end{remark}

\begin{remark}\label{rmk:blowupcurvature}
    The norm square $|\text{Riem}_g|^2$ of the Riemann curvature tensor of a Gibbons-Hawking metric $g$ with harmonic potential $f$ is given by the formula
    \[
    |\text{Riem}_g|^2 = \frac{1}{4} \Delta \Delta f^{-1} 
    \]
    where $\Delta$ is the usual Laplacian of $\R^3$. It follows from this that in the context of Theorem \ref{thm:aklimits} the curvature of the ALE metrics $g_t$ along the section $\sigma$ satisfies a bound of the form
    \begin{equation}
        |\text{Riem}_{g_t}(\sigma(t))|^2 = O(|t|^{-C})
    \end{equation}
    for uniform $C>0$ as $t\to 0$.
\end{remark}

\subsection{Towards a global construction}

The above local discussion for the bubblings for $A_k$-singularities should similarly describe what happens when we consider a holomorphic family of compact varieties forming such singularities. In particular, by combining the gluing techniques developed in the work of Biquard and Rollin for cscK metrics \cite{biquardrollin} and Ozuch \cite{ozuch} multiscale analysis of non-collapsed limits of Einstein $4$-manifolds \cite{ozuch}, we expect one should be able to prove the following result:

\begin{conjecture}
    Let $\pi:\mathcal{X} \rightarrow \Delta$ a smoothing of a KE orbifold $(X_0, \omega_0)$ with $A_k$-singularities and $Aut(X_0)$ discrete. Then nearby $X_t$ admits KE metrics whose full multiscale bubble tree can be recovered just from local algebraic data of the given family. More precisely, by considering the curve induced by the family in the versal deformation space of the $A_k$-singularity, and by varying sections of the family passing through the singularity, the tree of pointed Gromov-Hausdorff limits at a given singularity matches the local description given in Theorem \ref{thm:aklimits}.
\end{conjecture}

\begin{remark}
    The $Aut(X_0)$ conditions is not essential: the important aspect is that the smoothings admit KE metrics. Moreover, the above conjecture, easily extend (at least) to orbifolds admitting $\mathbb{Q}$-Gorenstein smoothings, as well as to the more general constant scalar curvature (cscK) case. 
\end{remark}

\subsection{Logarithmic situation} \label{Log2}
We now describe a simple example in the $2$ dimensional logarithmic situation when some new phenomenon occurs, suggesting that simple algebraic rescalings are not enough to describe the bubbling. This was pointed out first in \cite[p.21]{sun} for the absolute case and here we consider a logarithmic analogue.

We consider a family of smooth curves $C_t$ for $t \neq 0$ (see Equation \eqref{eq:family}) that develop a cuspidal singularity $\{w^2 = z^3\}$ at $t=0$. Moreover, we fix the cone angle parameter $\beta$ to be equal to $5/6$, which corresponds to the \emph{strictly semi-stable} case. In order to explain where the value $\beta=5/6$ comes from, we begin by reviewing the basics of theory of stability for klt pair singularities in this case.

\begin{lemma}
Consider the cuspidal curve $C = \{w^2 = z^3\} \subset \C^2$. 
Then the pair $(\C^2, b \cdot C)$ with $b=1-\beta \in (0,1)$ is klt if and only if the cone angle parameter $\beta$ belongs to the interval $(1/6,1)$.  
\end{lemma}

\begin{proof}
    The pair is klt if and only the integral
    \begin{equation}\label{eq:integral}
    I = \int_{B_1} |w^2-z^3|^{-2b} d\mu    
    \end{equation}
    is finite, where $B_1 \subset \C^2$ is the unit ball and $d\mu$ is the standard Lebesgue measure. In order to evaluate this integral we fix $\lambda \in (0,1)$, say $\lambda=1/2$, and consider the action $\lambda \cdot(z,w) = (\lambda^2 z, \lambda^3 w)$. The function $|w^2-z^3|^{-2b}$ is clearly integrable on the annulus $A_0 = B_1 \setminus B_2$ and we let $a_0 = \int_{A_0} |w^2-z^3|^{-2b} d\mu$. For $k \geq 1$ we let $A_k = \lambda^k \cdot A_0$ and $a_k = \int_{A_k} |w^2-z^3|^{-2b} d\mu$, so the integral \eqref{eq:integral} is finite if and only if the series $\sum a_k$ converges and $I = \sum_{k \geq 0} a_k$. On the other hand, by changing variables $\tz = \lambda^2 z, \tw = \lambda^3 w$ we see that $a_k = \lambda^{ck} a_0$ where $c = 10 - 12b$. Therefore $\sum a_k$ is a geometric series which converges if and only if $c >0$, equivalently $b < 5/6$.  
\end{proof}

Now consider the natural $\C^*$-action which preserves the curve $C = \{w^2 = z^3\}$ where $\lambda \in \C^*$ acts by
\begin{equation}\label{eq:action}
    \lambda \cdot (z, w) = (\lambda^2 z, \lambda^3 w) .
\end{equation}

\begin{lemma}
Let $\beta \in (1/6, 1)$ so that the pair $(\C^2, b \cdot C)$ with $b=1-\beta$ is klt. Then the pair is semistable with respect to the action \eqref{eq:action} if and only if $\beta \leq 5/6$.     
\end{lemma} 

\begin{proof}
    The quotient of the pair $(\C^2, b \cdot C)$ by the action \eqref{eq:action} is a sphere with three marked points $(\CP^1, \Delta)$ with $\Delta = b_0 \cdot 0 + b_{\infty} \cdot \infty + b \cdot 1$ where $0$ and $\infty$ are the orbifold locus of the quotient map and $b_0 = 1/2, b_{\infty} = 2/3$. The pair $(\C^2, b \cdot C)$ with the action \eqref{eq:action} is stable if and only if $(\CP^1, \Delta)$ is stable. On the other hand, $(\CP^1, \Delta)$ is semi-stable if and only if the coefficients $b_0, b_{\infty}, b$ satisfy the (closed) triangle inequalities. The inequalities $b \leq 1/2 + 2/3$ and $1/2 \leq 2/3 + b$ are always satisfied; therefore  $(\CP^1, \Delta)$ is semi-stable if and only if $2/3 \leq 1/2 + b$, i.e. $b \geq 1/6$.
\end{proof}

The local behaviour of a KE metric $g_{KE}$ with cone angle $2\pi\beta$ along a cuspidal curve $C = \{w^2 = z^3\}$ depends on the cone angle parameter $\beta$ as follows.

\begin{itemize}
\item \textbf{Stable case.} If $\beta \in (1/6,5/6)$ then there is a flat K\"ahler cone metric $g_C$ on $\C^2$ with cone angle $2\pi\beta$ along $C$. We can write
$g_C = d\rho^2 + \rho^2 \Bar{g}$ where $\rho>0$ is the intrinsic distance to the origin and $\Bar{g}$ is a metric on the $3$-sphere with constant sectional curvature $1$ and cone angle $2\pi\beta$ along the trefoil knot $C \cap S^3$. It is proved in \cite{deborbonspotti} that the metric $g_{KE}$ is asymptotic at a polynomial rate to $g_C$, i.e. in suitable local coordinates $|g_{KE} - g_C|_{g_C} = O(\rho^\mu)$ for some $\mu>0$ as $\rho \to 0$.

\item \textbf{Unstable case.}
If $\beta \in (5/6, 1)$ then \cite{dbe} produce a Calabi-Yau metric $g_{CY}$ in a neighbourhood of $0 \in \C^2$ with cone angle $2\pi\beta$ along $C \setminus \{0\}$ whose tangent cone at the origin is equal to the product $\C \times \C_{\gamma}$ where $\gamma = 2\beta -1$. Following \cite{chiuszek} we expect that $g_{KE}$ has polynomial convergence to (a multiple) of $g_{CY}$ at the level of potentials.

\item \textbf{Strictly semistable case.}
If $\beta = 5/6$ then the tangent cone at $0$ of $g_{KE}$ as predicted algebraically by the theory of normalized volumes (see \cite[Section 7]{deborbonspotti}) is the product $\C \times \C_{\gamma}$ where $\gamma = 2\beta -1$, same as in the unstable case. However, in this case $g_{KE}$ is only expected to converge to its tangent cone at a logarithmic rate.
\end{itemize}

\begin{remark}
    We can provide a geometric interpretation for the change of tangent cone as $\beta \to 5/6$.
    In the stable case $\beta \in (1/6, 5/6)$ there is a metric $g_{\CP^1}$ on the $2$-sphere with $3$ cone points of total angle $2\pi/3, \pi$, and $2\pi\beta$ which lifts through the Seifert fibration  $S^3 \to \CP^1$ given by the quotient projection by the action \eqref{eq:action} to the constant curvature $1$ metric $\Bar{g}$ on $S^3$ that is the link of the tangent cone $g_C = d\rho^2 + \rho^2 \Bar{g}$. The metric $g_{\CP^1}$ is the double of a spherical triangle with angles $\pi/3, \pi/2$, and $\pi\beta$. As $\beta \to 5/6$ the triangle converges to a spherical bigon with angles $\pi/3$. Taking the double of the bigon and lifting through the Seifert fibration gives a constant curvature $1$ metric on the $3$-sphere with cone angle $2\pi\gamma$ (with $\gamma=2/3$) along the circle lying over the $1/2$ orbifold point which is the link of the tangent cone $\C \times \C_{\gamma}$ when $\beta=5/6$.
\end{remark}


The Euler vector field $e$ and the metric dilations $d_{\lambda}$ for $\lambda>0$ of the tangent cone are as follows.

\begin{itemize}
\item \textbf{Stable case.} If $\beta \in (1/6, 5/6)$ then
\begin{equation}\label{eq:euler1}
e = \frac{2}{\alpha} z \frac{\p}{\p z} + \frac{3}{\alpha} w \frac{\p}{\p w} \hspace{2mm} \mbox{ and } \hspace{2mm}
d_{\lambda} (z,w) = (\lambda^{2/\alpha} z, \lambda^{3/\alpha}w) ,
\end{equation}
where $\alpha = 3\beta - (1/2)$. The Euler vector field is tangent to $C=\{w^2=z^3\}$ and the metric dilations preserve this curve.   

\item \textbf{Unstable case.}
If $\beta \in (5/6, 1)$ then
\begin{equation}\label{eq:euler2}
e = z \frac{\p}{\p z} + \frac{1}{\gamma} w \frac{\p}{\p w} \hspace{2mm} \mbox{ and } \hspace{2mm}
d_{\lambda} (z,w) = (\lambda z, \lambda^{1/\gamma}w) .
\end{equation}
This time the Euler vector field of the cone is \emph{not} tangent to the cusp $C=\{w^2=z^3\}$ and the metric dilations of the cone degenerate the curve $C$ into the line $\{w=0\}$. More precisely, if we let $z = \lambda \tz$, $w = \lambda^{1/\gamma} \tw$ then the curve $C$ is given by $\tw^2 = \lambda^{3-(2/\gamma)} \tz^3$. The condition $\beta>5/6$ is equivalent to $3-(2/\gamma)>0$, hence the preimage of $C$ under the action $d_{\lambda}$ given by Equation \eqref{eq:euler2} converges to the double line $\{\tw^2=0\}$ as $\lambda \to 0$.  

\item \textbf{Strictly semistable case.}
If $\beta=5/6$ then the two vector fields and actions given by Equations \eqref{eq:euler1} and \eqref{eq:euler2} agree. The cuspidal curve $C$ is invariant under the action, however the tangent cone is $\C \times \C_{\gamma}$ and the degeneration of the pair $(\C^2, C)$ to its tangent cone is through a different (non-canonical) action, such as $\lambda \cdot (z,w) = (\lambda z, \lambda w)$ for example. 
\end{itemize}

From now on we focus on the strictly semistable case and fix $\beta=5/6$.
Consider a sequence of KE metrics with cone angle $2\pi\beta$ along a family of smooth curves $C_t$ that develop a cuspidal singularity at $t=0$, say given by
\begin{equation}\label{eq:family}
C_t = \{w^2 = z^3 + tz\} .
\end{equation}

We rescale the family of curves \eqref{eq:family} using the weights $(1,3/2)$ of $\C \times \C_{\gamma}$ (where $\gamma = 2\beta-1 = 2/3$) which is the tangent cone at the origin of $g_0$, as in the case of $A_k$-singularities (see Remark \ref{rmk:ansatz}). We let $z = t^c \tz$, $w = t^{3c/2} \tw$ and take $c=1/2$ to obtain in the limit as $t \to 0$ the curve
\begin{equation}\label{eq:ssbubble}
    \tw^2 = \tz^3 + \tz .
\end{equation}

\begin{expectation}
    There is no CY metric on $\C^2$ with cone angle $2\pi\beta$ along the curve \eqref{eq:ssbubble} and whose tangent cone at infinity is $\C \times \C_{\gamma}$, where $\beta=5/6$ and $\gamma=2\beta-1$.
\end{expectation}

Indeed, if such a metric exists then rescaling the coordinates by the weights $(1,3/2)$ of the tangent cone at infinity $\tz \mapsto \lambda \tz$, $\tw \mapsto \lambda^{3/2} \tw$ degenerates the curve \eqref{eq:ssbubble} to
\[
\tw^2 = \tz^3 + \lambda^{-2} \tz .
\]
which converges to $\tw^2 = \tz^3$ as $\lambda \to \infty$. By the results of \cite{sunzhang} we would expect that there is a CY metric on $\C^2$ with cone angle $2\pi\beta$ along cuspidal curve $\tw^2 = \tz^3$ whose tangent cone at infinity is $\C \times \C_{\gamma}$. However, the tangent cone at the origin should also be $\C \times \C_{\gamma}$ and by Bishop-Gromov the metric must be a cone but this is impossible as jumping occurs.

On the other hand, if we rescale the family \eqref{eq:family} by $z = t \tz$, $w = t \tw$ we obtain $\tw^2 = t \tz^3 + \tz$ which converges to
\begin{equation}\label{eq:parabola}
    \tw^2 = \tz 
\end{equation}
as $t \to 0$. 

\begin{expectation}
    The bubble of the family \eqref{eq:family} is a CY metric on $\C^2$ with cone angle $2\pi\beta$ along the parabola \eqref{eq:parabola} and tangent cone $\C \times \C_{\gamma}$ at infinity.
\end{expectation}

If we rescale the parabola by the weights of the tangent cone $\tz \mapsto \lambda \tz$, $\tw \mapsto \lambda^{3/2} \tw$ we obtain the family $\tw^2 = \lambda^{-2} \tz$ which converges to the double line $\tw^2 =0$ as $\lambda \to \infty$, as wanted. 



\section{Towards an algebro-geometric picture of bubbling and possible multiscale K-moduli compactifications $\overline{\mathcal{M}}^\lambda$}\label{sec:highdim}

In this last section we  will first propose a tentative more invariant algebraic picture for detecting the bubbles, slightly reinforcing Sun's conjecture in \cite{sun}. Finally, we will formulate a more general framework to study bubbling, by speculating about the existence of what we call \emph{multiscale K-moduli compactifications} $\mathcal{T M}^K$, that is  birational modifications of the K-moduli spaces which parameterize algebraic spaces supporting all the possible multiscale bubble limits of non-collapsing KE metrics.

\subsection{Algorithm for the algebro-geometric detection of the metric bubble tree of a $1$-dimensional family}

Suppose we have a flat holomorphic family $$\pi:\mathcal{X} \rightarrow \Delta$$ over the complex disc of non-collapsing KE spaces (note that in what follows smoothness is not really required), and take a section $\sigma: \Delta \rightarrow \mathcal{X}$ for the family such that $\sigma(0) \in \mbox{Sing}(X_0)$. 

Working locally, and ignoring convergence issues, we can restrict to consider our family as defined on the local ring $\C[[t]]$ of formal power series, and the section to be a ring morphism $s: R \rightarrow \C[[t]]$ from a finitely generated $\C[[t]]$-algebra $R$. Note that $\sigma(0)\in \mathcal{X}$ is then the image of the unique maximal ideal of  $\C[[t]]$ via the map of schemes induced by $s$, so we can think $\sigma(0)=x \in \emph{Spec}(R)$.  

Now consider the non-Archimedean link $\mbox{Val}_{R_x}$, that is maps $v: R_x \rightarrow \mathbb{R} \cup \{+\infty\}$, satisfying the usual axioms for  valuations on the local ring $R_x$.
The claim is that there is the following inductive procedure (terminating in finite steps), mirroring Sun's differential geometric termination of bubbles. 
Following \cite{sun} we call \emph{minimal bubble} a non-cone rescaled limit whose tangent cone at infinity is equal to the tangent cone at the singularity.

There is a valuation $v_1 \in \mbox{Val}_{R_x}$ such that the following properties hold:
\begin{enumerate}
    \item \emph{Weighted family}. The graded ring $R^1:= \bigoplus_j \frac{I_{k_j}^1}{I_{k_{j+1}}^1}$ is a finitely generated $\C[t]$-algebra, with ideals $I_{k_j}=\{f\in R_x, v_1(f)\geq k_j\}$. Then
    $$\mathcal{W}^1=\mbox{Spec}(R^1) \rightarrow \C$$ 
    is a $ \C^*$-equivariant \emph{weighted family} such that the general fiber $W^1_t$ is a (minimal) \emph{weighted bubble}, a negative weight deformation (see \cite{CH}) of the special fiber $W^1_0$ the \emph{weighted tangent cone}.   
    \item \emph{Bubbling family}. There exists a non-canonical $\C^\ast$-equivariant degeneration of this family to a new equivariant family $$\mathcal{B}^1 \rightarrow \C$$ (called the \emph{bubbling family}) where the general fiber $B^1_t$ is the (minimal) \emph{algebraic bubble}, a negative weight deformation of the central fiber $B^1_0$, the \emph{algebraic tangent cone}. Note that the space of negative weighted deformation should always be finite dimensional, even for a cone with non-isolated singularities.    
    
    \item \emph{Iteration}: By specialization, the original family with its section gives rise to two new families (isomorphic to the original family away from zero) $\Tilde{\mathcal{X}}_t^1$ and $\Tilde{\Tilde{\mathcal{X}}}_t^1$, having the (weighted) bubbles in the fiber over the origin. We are now back to the beginning: there should exists a valuation such that we can repeat the above steps $(1)(2)$, given a new local weighted family $\mathcal{W}^2$ and a new bubbling family $\mathcal{B}^2$, and new iteration families $ \Tilde{\mathcal{X}}_t^2$ and $\Tilde{\Tilde{\mathcal{X}}}_t^2$. The process will repeat a finite number of steps until we reach the situation in which the algebraic tangent cone coincides with $\mathbb{C}^n$. We refer to the corresponding algebraic bubble in the step before that as the \emph{deepest bubble}.
\end{enumerate}

\begin{figure}[h]
    \centering
    \scalebox{.5}{
\begin{tikzpicture}
	\coordinate (a) at (0,0);
	\coordinate (b) at (0,4);
	\coordinate (c) at (0,8);

	\filldraw[] (a) circle (4pt);
	\filldraw[] (b) circle (4pt);
	\filldraw[] (c) circle (4pt);
	
	\draw [-{To[scale=2]}] (a) -- (0,2);
	\draw [-{To[scale=2]}] (b) -- (0,6);

	\draw (a) -- (b);
	\draw (b) -- (c);

	\draw [-{To[scale=2]}] (-6,0) -- (-3,0);
	\draw [-{To[scale=2]}] (-6,4) -- (-3,4);
	\draw [-{To[scale=2]}] (-6,8) -- (-3,8);
	
	\draw [-{To[scale=2]}] (6,0) -- (3,0);
	\draw [-{To[scale=2]}] (6,4) -- (3,4);
	\draw [-{To[scale=2]}] (6,8) -- (3,8);
	
	\draw (-3,0) -- (3,0);
	\draw (-3,4) -- (3,4);
	\draw (-3,8) -- (3,8);
	
	\node[scale=2] at (-6.5, 8.2) {\(\mathcal{B}\)};
	\node[scale=2] at (-6.5, 4.2) {\(\mathcal{W}\)};
	\node[scale=2] at (-6.5, .2) {\(\mathcal{X}\)};
	
	\node[scale=1.5] at (.4, 7.7) {\(B_0\)};
	\node[scale=1.5] at (.4, 3.7) {\(W_0\)};
	\node[scale=1.5] at (.4, -.3) {\(X_0\)};
	
	\node[scale=1.5] at (3.4, 7.6) {\(B_t\)};
	\node[scale=1.5] at (3.4, 3.6) {\(W_t\)};
	\node[scale=1.5] at (3.4, -.4) {\(X_t\)};	
\end{tikzpicture}    
}
\caption{\(2\)-step degeneration of the family \(\mathcal{X}\) to the bubble family \(\mathcal{B}\) for which the general fibre is the minimal bubble \(B\). The central fibres \(X_0 \to W_0 \to B_0\) give the usual \(2\)-step degeneration of the germ \((X_0, x_0)\) to its tangent cone.} 
\label{fig:2stepdeg}
\end{figure}
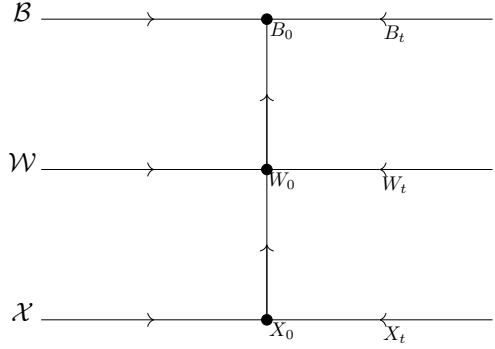

The weighted and bubbling families of the above items $(1)(2)$ are represented in Figure \ref{fig:2stepdeg}. 

\begin{remark}
 Note that, even if we start with a family with compact fibers and only defined locally in $t$, the rescaled local families $\mathcal{W, B}$ of items $(1)(2)$ are going to be affine and defined on all $\C$.  Regarding the crucial step $(3)$, the process should be described by a weighted blow-up of the original family, giving two families still over $\C[[t]]$, say $\Tilde{\mathcal{X}}_t$ and $\Tilde{\Tilde{\mathcal{X}}}_t$,  which  have the bubbles $W_1$ and  the actual metric bubble $B_1$ (as components) in the fiber over the origin.  In particular, $W_1$ and $B_1$ will give (divisorial) valuations corresponding to special points in the Berkovich analytification (over the given section) of the original family $\pi:\mathcal{X} \rightarrow \Delta$ (of course such valuations are expected to be linked to the one given the weighted bubbling local families of steps $(1),(2)$).

\end{remark}

Motivated by the analysis in the previous sections,  we then propose the following.

\begin{expectation}[Relation to differential geometry] The outlined algebraic algorithm above compute all possible metric bubbles (pointed Gromov-Hausdorff limits) of a non-collapsing family of KE metrics along a section $\sigma$. More precisely,

\begin{itemize}
\item  All possible metric bubbles do not depend on the particular sequence of $t_k\rightarrow 0$, but only on the holomorphic section $\sigma$.

\item The algebraic bubbles $B^i_1$ can be identified with all the (non-cone) rescaled metric bubbles, and hence they admit (possibly singular) Calabi-Yau metrics asymptotic to the Calabi-Yau cones $B^i_0$. 

\item The algebraic finite step iterations of the algorithm correspond to the finite number of possible metric bubbles, as described in  \cite{sun}.
\end{itemize}

\end{expectation}

\begin{remark}
    In the above context, for the given flat family $\mathcal{X} \to \Delta$ of non-collapsing KE metrics and an algebraic section $\sigma: \Delta \to \mathcal{X}$, we expect that the curvature of the manifolds $X_t$ at the points $\sigma(t)$ for $t \neq 0$ blows-up at a polynomial rate of the form $|\mbox{Riem}_{g_t}(\sigma(t))| \leq C |t|^{-C}$ for some $C>0$, corresponding to the finite step termination of the algorithm. See Remark \ref{rmk:blowupcurvature} for the case of $2$-dimensional $A_k$ singularities.
\end{remark}



Of course, the crucial aspect of the conjectural picture is to determine a priori via algebraic geometry which valuations will make such identification between algebraic and differential geometric rescalings to hold. As our experiments show (see also Section \ref{sec:ak}), it seems that the valuations are related to the ones giving the normalized volume for the klt singularities in the central fiber $X_0$ (that is, differential geometrically, the ones computing the metric tangent cone) via Li's normalized volume.

\begin{remark}Evocatively, we can think to the above algorithm to provide a \emph{geometric algebraic expansion} of the space $\mathcal{X}$ around the section $\sigma$:
$$\mathcal{X}_{\sigma} \cong \sum_i B^i_1 t^{v_i},$$
where $v_i$ are the valuations computing the tangent cones. 
\end{remark}


More generally, one can ask what happens by letting the sections vary. Our previous examples in Sections \ref{sec:1dim} and \ref{sec:2dim} suggest the following picture: 

\begin{expectation}
    
There are equivalence relations on the set of sections through a point $p\in X_0$ so that two equivalent sections induce the same bubbles up to certain scale. Thus, by varying the sections we expect to be able to algebro geometrically detect the full metric bubble tree for the degenerating family $\pi:\mathcal{X} \rightarrow \Delta$ at a given singularity. Schematically, this looks like a $1$-dimensional tree $\mathcal{T}$, with root given by the space $X_0$ itself and leaves given by all the possible deepest bubbles (removing all flat limiting spaces, or, more generally, if we consider singular families with section $\sigma$ all taking value at singular points, all tangent cones at the generic point of $\sigma$). The segments correspond to the tangent cones.  A choice of a section will determine a non-branching  path from the root to a leaf, as indicated in Figure \ref{fig:polytreepath}.

\end{expectation}

We now discuss some conceptual examples (see also previous section $3.5$ for a logarithmic situation). Steps $(1)(2)$ above can be simply obtained by rescaling a given family centered at at the singularity of the central fiber, where one uses in the vertical directions weights of the tangent cone, and a suitable rescaling on the horizontal $t$ variable, in analogy to the usual computations for the metric tangent cone. Next examples (we will focus on step (3) and its iterations) are, for simplicity, given only locally, and they are still conjectural (but the rescalings we consider are in analogy to the ones that works in low dimension as described in previous sections).

\subsection{$A_k$-singularities}\label{sec:ak}
Let $n \geq 3$ and consider the $n$-dimensional $A_k$-singularity given by
\[
A_k = \{x_1^2 + \ldots + x_n^2 = x_0^{k+1}\} \subset \C^{n+1} .
\]
Suppose that the singularity is strictly unstable, this is
\begin{equation}\label{eq:akunst}
    k+1 > 2 \frac{n-1}{n-2} .
\end{equation}

If the unstable condition \eqref{eq:akunst} holds then the weighted tangent cone $W$ as well as the tangent cone at the origin $C(Y)$ are equal to the product of \(\C\) with the \((n-1)\)-dimensional $A_1$-singularity
\[
W = C(Y) = \C \times A_1^{(n-1)}  .
\]
The weights of the tangent cone are
\begin{equation}
    \textbf{w} = \left( 1, \frac{n-1}{n-2}, \ldots, \frac{n-1}{n-2} \right) .
\end{equation}

The space of negative weight deformations of $C(Y)$ is made of hypersurfaces $B \subset \C^{n+1}$ of the form
\begin{equation}\label{eq:negweight}
B = \{ x_1^2 + \ldots + x_n^2 = P(x_0) \hspace{2mm} | \hspace{2mm} \deg(P) \leq 3 \text{ if } n = 3, \hspace{2mm} \deg(P) \leq 2 \text{ if } n \geq 4\} .    
\end{equation}
In particular, any bubble arising from the formation of an unstable $A_k$-singularity should be of the form given by Equation \eqref{eq:negweight}.

\begin{example}
    Take $(n, k) = (3, 4)$ and consider the family
    \begin{equation}\label{eq:akfamily}
    X_t = \{ x_1^2 + x_2^2 + x_3^2 = x_0 (x_0 + t) (x_0 + t + t^2) (x_0 + t^3) (x_0 + t^3 + t^4)  \} .
    \end{equation}
    We describe rescaled limits centered at the zero section.
    The weights are $\textbf{w} = (1,2,2,2)$. Let $c>0$ and rescale $x_0 = t^c \tilde{x}_0$ and $x_i = t^{2c} \Tilde{x}_i$. Taking $c=2$ we obtain the family
    \begin{equation}\label{eq:rescfam}
    \Tilde{X}_t = \{ \tx_1^2 + \tx_2^2 + \tx_3^2 = \tx_0 (t \tx_0 + 1) (t \tx_0 + 1 + t) (\tx_0 + t) (\tx_0 + t + t^2)  \} .
    \end{equation}

The minimal bubble $B_{\alpha_1}$ is obtained by setting $t=0$ in Equation \eqref{eq:rescfam}
\[
B_{\alpha_1} = \{ \tx_1^2 + \tx_2^2 + \tx_3^2 = \tx_0^3 \} .
\]
The minimal bubble $B_{\alpha_1}$ is the total space of the $3$-dimensional $A_2$-singularity. The $A_2$-singularity is stable (Equation \eqref{eq:akunst} is strictly violated) and admits a Calabi-Yau cone metric found by Li-Sun \cite{lisun}. We expect that $B_{\alpha_1}$ should admit a Calabi-Yau metric asymptotic to the Li-Sun metric at the origin and whose tangent cone at infinity is the product $\C \times (\C^2/\mathbb{Z}_2)$.

Next we rescale the family \eqref{eq:rescfam} using the weights of the $A_2$-singularity. Let $c>0$ and write $\tx_0 = t^{2c} x_0'$, $\tx_i = t^{3c} x_i'$. Setting $c=1/2$ and taking $t \to 0$ we obtain that the second bubble is
\[
B_{\alpha_2} = \{ (x'_1)^2 + (x'_2)^2 + (x'_3)^2 = x'_0 (x'_0+1)^2 \} .
\]
The space $B_{\alpha_2}$ has an isolated $A_1$ singularity at $x_0' = -1$, $x_i'=0$ for $i=1, 2, 3$ and it is otherwise smooth.
We expect that $B_{\alpha_2}$ admits a Calabi-Yau metric asymptotic to the Stenzel cone metric at its singular point and whose tangent cone at infinity is the Li-Sun cone metric on the $A_2$-singularity. Finally, since we are rescaling at the zero section and the origin is a smooth point of $B_{\alpha_2}$, the next rescaled limit is flat $\C^3$ and the process terminates.  See Figure \ref{fig:aklimits}.
\end{example}

\begin{figure}
	\begin{center}	
 \scalebox{.4}{
	\begin{tikzpicture}[xshift=2cm]
	\draw (-2,4) -- (0,0) -- (2,4);
	\draw (-2,4) to[bend right] (2,4) to[bend right] (-2,4);
	\filldraw (0,0) circle (2pt);
	
	\draw[xshift=6cm] (-2,4) to[bend right] (0,0) to[bend right] (2,4);
	\draw[xshift=6cm] (-2,4) to[bend right] (2,4) to[bend right] (-2,4);
	\filldraw[xshift=6cm] (0,0) circle (2pt);
	
	\draw[xshift=12cm] (-2,3) -- (0,0) --(2,3);
	\draw[xshift=12cm] (-2,3) to[bend right] (2,3) to[bend right] (-2,3);
	\filldraw[xshift=12cm] (0,0) circle (2pt);
	
	\draw[xshift=18cm] (-2,4) to[out=-90, in=180] (0,0) to[out=0, in=-135] (1,1.5) to[out=45, in=135] (1.5,1) -- (2,4);
	\filldraw[xshift=18cm] (0,0) circle (2pt);
	\draw[ultra thick, red, xshift=18cm] (1.5, 1) circle (1pt);
	\draw[xshift=18cm] (-2,4) to[bend right] (2,4) to[bend right] (-2, 4);
	
	\draw[xshift=24cm, rounded corners] (-2, 4) -- (-2,0) -- (2,0) -- (2,4);
	
	\draw (-2,-2) -- (26,-2);
	\draw[thick] (-2,-1.6) -- (-2,-2.4);
	\draw[thick] (6,-1.6) -- (6,-2.4);
	\draw[thick] (18,-1.6) -- (18,-2.4);
	
	\node[scale=2] at (-2, -3) {\(0\)};
	\node[scale=2] at (6, -3) {\(\alpha_1\)};
	\node[scale=2] at (18, -3) {\(\alpha_2\)};
	\node[scale=2] at (2, -1.5) {\(\mathbb{C} \times (\mathbb{C}^2 / \mathbb{Z}_2)\)};
	\node[scale=2] at (12, -1.5) {Li-Sun};
	\node[scale=2] at (24, -1.5) {\(\mathbb{C}^3\)};
	\node[scale=2] at (6, 2) {\(B_{\alpha_1}\)};
	\node[scale=2] at (18, 2) {\(B_{\alpha_2}\)};
	
	\end{tikzpicture}	
 }
\caption{The rescaled limits $|t|^{-\alpha} (X_t, x_t)$ where $X_t$ is a non-collapsed family of KE manifolds locally defined by Equation \eqref{eq:akfamily} and the base points $x_t$ are given by the zero section. }
\label{fig:aklimits}
\end{center}
\end{figure}
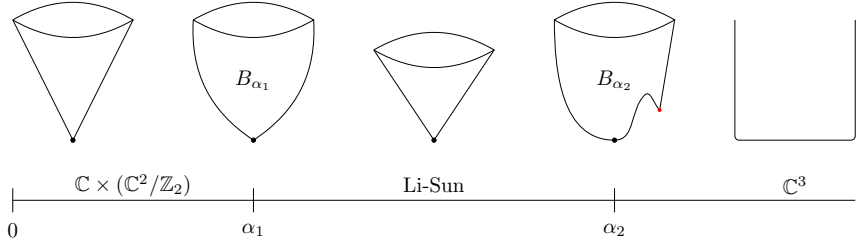

The next two examples illustrate the dependence of the bubbles on the section.

\begin{example}[Non-isolated singularities]
Consider the family
\begin{equation}
    X_t = \{x_1^2 + x_2^2 + x_3^2 = t x_0\} .
\end{equation}
Let $\sigma (t) = (x_0(t), x_1(t), x_2(t))$ be the section with $\sigma(0)=0$. In the generic case that the first derivative $x_0'(0) \neq 0$  the bubble is the product of Eguchi-Hanson with $\C$. In the special case that $x_0'(0)=0$ the bubble is a metric on $\C^3$ with tangent cone at infinity $(\C^2/\mathbb{Z}_2) \times \C$ as found in \cite{yli, conlonrochon, szek}.
\end{example}

\begin{example}
Suppose that the unstable condition \eqref{eq:akunst} holds and consider the family
\begin{equation}
    X_t = \{x_1^2 + \ldots + x_n^2 = \prod_{j=1}^{k+1} \left( x_0 - a_j t \right) \}
\end{equation}
where $a_j \in \C$ and $a_i \neq a_j$ for $i \neq j$. If the section $s(t) = (x_0(t), \ldots, x_n(t))$ with $s(0) = 0$ is generic, in the sense that $x_0(t) = at + \text{(h.o.t)}$ with $a \neq a_j$ for all $j$, then we obtain the product of the Stenzel metric on $\{x_1^2+\ldots+x_n^2=1\}$ with $\C$ as a bubble. However, if the section if of the form $x_0(t) = a_j t + \text{(h.o.t)}$  for some $1 \leq j \leq n$ then we expect the bubble to be one the Calabi-Yau metrics on $\C^n$ with splitting tangent cone at infinity $A_1^{n-1} \times \C$ found in \cite{szek}.   
\end{example}

\begin{remark}
All these local examples can be easily realised as local analytic singularities' formations of families of \emph{compact} KE varieties. In particular, establishing an a-priori algebro-geometric detection of bubbling will give a construction of asymptotically conical CY metrics (also singular and with splitting tangent cones at infinity, as described in the examples above) \emph{directly as bubble limits of compact KE manifolds}, hence avoiding a direct and subtle weighted analysis of the Monge-Amp\`ere equation on (possibly singular) affine spaces but instead relaying on some sort of \emph{multiscale moduli continuity method}. More concretely, let's illustrate this philosophy  with a toy example, showing (and this is a rigorous deduction, thanks to the already existing theory) the emergence of the Stenzel's AC Calabi-Yau metric on the affine variety  $x_1^2 + \ldots + x_n^2 = 1$ \emph{directely} as a bubble limit.

\end{remark}

\begin{proposition} Let $\pi:\mathcal{X} \rightarrow \Delta$ be a non-collapsing family of compact  KE manifolds, and assume that $p\in X_0$ is an $A_1$-singularity. Then the Stenzel CY metric on $x_1^2 + \ldots + x_n^2 = 1$ appears as a minimal bubble (no gluing required!).
\end{proposition}

\begin{proof} By the established deep general theory of convergence, we know that such metrics GH convergence to a singular KE metric on $X_0$ whose metric tangent cone is the Calabi's ansatz cone metric on the canonical bundle of the KE Fano on the smooth quadric hypersurface. Moreover, by \cite{sun} we know that there should be a rescaled minimal bubble that has such cone as its tangent cone at infinity, and it is a negative weight deformation of it \cite{CH}. However, versal deformations of an $A_1$ singularity are trivial: there is only \emph{a} smoothing. Hence, there must be an AC CY metric on the affine space $x_1^2 + \ldots + x_n^2 = 1$ with that tangent cone at infinity (which has to be precisely the AC Stenzel's metric by uniqueness). 
\end{proof}

\begin{remark}
For explicit examples, just take a family of smooth projective hypersurfaces in $\C \mathbb{P}^N$ of sufficiently higher degree degenereting to a klt variety $X_0$ with an $A_1$ singularity (the generic singularity at the boundary of their moduli space).    
\end{remark}

\subsection{Towards multiscale K-moduli compactifications?} \label{mK}
The previous study of the Deligne-Mumford/Mostow compactifications for the log $1$-dimensional case in Section \ref{sec:1dim}, combined with the study of deformations of $A_k$-singularities in Section \ref{sec:ak}, points towards the following highly speculative picture, which we will now roughly describe in its more optimistic formulation (thus our ``expectations" should be tuned appropriately). In short, we can ask the following:

\begin{question} Are there compactifications of moduli spaces which parametrize bubbles as well?
\end{question}

Let's elaborate a bit more. First one can ask what happens if we let the family $\pi:\mathcal{X} \rightarrow \Delta $ vary while fixing  the central fiber $X_0$. A family $\pi:\mathcal{X} \rightarrow \Delta $ can be considered as a curve in the non-collapsing part of the K-moduli compactification $\overline{\mathcal{M}}$ passing at $[X_0]\in \overline{\mathcal{M}}$. From this point of view, a section is just a lift of that curve to the universal family $\mathcal{U} \rightarrow \mathcal{M}$  (assume we have one for simplicity). We should have a recipe to associate to a one dimensional family the sets of minimal bubbles by varying $p\in \mbox{Sing}(X_0)$. Note that two different families may give the same minimal bubbles at $p$ (e.g., this could happen if the reduction mod $t^2$ of the families give the same tangent at $T_{[X_0]}\mathcal{M}$, but the general picture may be more complicated than that).

Each minimal bubble $B^1$ comes with an equivariant \emph{negative weight} $\C^\ast$ degeneration to its associated tangent cone.  The space of such $\mathbb{Q}$-Gorenstein negative weight deformations  $Def_\mathbb{Q}^-(B_0^0)$ should be some finite dimensional space. Hence, at last in good situations (as we have in the example of $A_k$-singularities in dimension two), we can have a subset of $Z^1\subseteq \mathbb{P} Def_\mathbb{Q}^-(B_0^0)$ to be a \emph{moduli space of first minimal bubble at $p\in X_0$}. In concrete situations we expect $Z$ to depend on the all $X_0$, and not only from obstructions to the local deformation of the singularity. This is related to the fact that there are local to global obstructions to the deformation of singularities (remarkably, this is not the case for del Pezzo surfaces, as shown by Hacking and Prokhorov \cite{hackingprokhorov}). This is true even when jumping phenomena do not occur. Thus, motivated by the bubbling interpretation of the Deligne-Mumford compactification in \ref{sec:1dim}, it is tempting to ask if there is some algebro-geometric moduli spaces (a birational modification of the  K-moduli) $\overline{\mathcal{M}}^{\lambda_1}$ parametrizing in its boundary the first minimal bubbles at the various singularities of $X_0 \in \partial \overline{\mathcal{M}}^{K}$ as discussed above. Moreover, by varying scales and (equivalence classes of) sections/lifts one could hope to find a set of birational modifications of K-moduli spaces parametrizing bubbles at \emph{different scales}  $\mathcal{T M}^K=\{\overline{\mathcal{M}}^{\lambda_i}\}$ (also, which scales are indeed realizable algebro-geometrically? All?).  It may be interesting to investigate if something like that is happening on some simple examples (e.g., K-moduli of del Pezzo surfaces or K3 surfaces).


\bibliographystyle{plainurl}
\bibliography{bub}

\end{document}